\documentclass[12pt,reqno]{amsart}
\usepackage{amsmath,amssymb}
\usepackage[colorlinks=true,linkcolor=magenta,citecolor=blue]{hyperref}
\usepackage[margin=1in]{geometry}
\usepackage{enumitem} % for enumerate environment
\usepackage{mathabx} 
\usepackage{xcolor}
\title[Global classical solutions to quadratic systems]{Global classical solutions to quadratic systems with mass control in arbitrary dimensions}
\author{Klemens Fellner, Jeff Morgan, Bao Quoc Tang}
\address{Klemens Fellner \hfill\break
	Institute of Mathematics and Scientific Computing, University of Graz, Heinrichstrasse 36, 8010 Graz, Austria}
\email{klemens.fellner@uni-graz.at}
\address{Jeff Morgan \hfill\break
Department of Mathematics, 
University of Houston, Houston, Texas 77004, USA}
\email{jmorgan@math.uh.edu}
\address{Bao Quoc Tang \hfill\break
	Institute of Mathematics and Scientific Computing, University of Graz, Heinrichstrasse 36, 8010 Graz, Austria}
\email{quoc.tang@uni-graz.at} 

\newcommand{\R}{\mathbb R}	
\renewcommand{\O}{\Omega}

\newcommand{\pa}{\partial}
\newcommand{\wh}{\widehat}
\newcommand{\wt}{\widetilde}
\newcommand{\vr}{\varrho}

\newcommand{\mb}{\boldsymbol}
\newcommand{\nx}{\nabla_{\!x}}
\newtheorem{theorem}{Theorem}[section]
\newtheorem{definition}{Definition}[section]
\newtheorem{lemma}{Lemma}[section]
\newtheorem{proposition}{Proposition}[section]

\newtheorem{remark}{Remark}[section]

\begin{document}

	\begin{abstract}
	The global existence of classical solutions to reaction-diffusion systems in arbitrary space dimensions is studied. The nonlinearities are assumed to be quasi-positive, to have (slightly super-) quadratic growth, and to possess a mass control, which includes the important cases as mass conservation and mass dissipation. Under these assumptions, the local classical solution is shown to be global and, in case of mass conservation or mass dissipation, to have $L^{\infty}$-norm growing at most polynomially in time. Applications include skew-symmetric Lotka-Volterra systems and quadratic reversible chemical reactions.
\end{abstract}
%\linenumbers

	\maketitle
	
	\tableofcontents
	
\section{Introduction and Main results}
In this paper, we study the global existence of classical solutions 
to a class of nonlinear reaction-diffusion systems.
Let $\Omega\subset\mathbb R^n$, $n\geq 1$, be a bounded domain and $u_i: Q_T:= \Omega\times (0,T) \to \R$, $i=1,\ldots, N$ be the $i$-th concentration. Denote by $u = (u_1, \ldots, u_N)$ the vector of concentrations. We consider the following reaction-diffusion system
	\begin{equation}\label{eq1}
		\begin{cases}
			\partial_t u_i - d_i\Delta u_i  = f_i(u), &\quad (x,t)\in Q_T,\\
			\nx u_i \cdot \nu = 0, &\quad (x,t)\in \partial\O\times (0,T),\\
			u_i(x,0) = u_{i0}(x), &\quad x\in\Omega,
		\end{cases}
	\end{equation}
%	\red{The proof seems not directly extended to the case of Dirichlet boundary condition. See the Proof of Lemma \ref{lem:main}.}
%\textcolor{blue}{What about other boundary conditions? Dirichlet? $\longrightarrow$ Remark after the main theorem.}
where $d_i > 0$ are the diffusion coefficients, and the domain, the initial data and the nonlinearities satisfy the following assumptions:
	\begin{enumerate}[label=(A\theenumi),ref=A\theenumi]
		\item\label{A1} (Smooth Domain) Either $\Omega = \R^n$ or $\Omega\subset \R^n$ is a bounded domain with smooth boundary $\pa\O$ (i.e. $\pa\O$ is of $C^\infty$ class), such that $\Omega$ lies locally on one side of $\partial\Omega$\footnote{Naturally, the zero-flux boundary condition is only considered in case of bounded domain.}. %and \blue{convex} boundary $\partial\Omega$. (\blue{The convex boundary is required to use the Li-Yau estimates for heat kernels.})
		\item\label{A2} (Bounded, Nonnegative Initial Data) For all $i=1\ldots N$: $u_{i0}\in L^{\infty}(\Omega)$ and $u_{i0}(x) \geq 0$ for a.e. $x\in\Omega$.
		\item \label{A3} (Mass Control) There exist $K_0\geq 0$, $K_1\in \mathbb R$ such that
\begin{equation*}
\sum_{i=1}^{N}f_i(u) \leq K_0 + K_1\sum_{i=1}^{N}u_i, \qquad \text{for all}\quad 
u\in \R^N_+:= [0,\infty)^N.	
\end{equation*}
		\item\label{A4} (Local Lipschitz and Quasi-positivity) For all $i=1,\ldots, N$, $f_i(u)$ is locally Lipschitz and 
\begin{equation*}
f_i(u) \geq 0, \qquad \text{for all}\quad u\in \R^N_+ \quad  \text{satisfying}\quad u_i = 0.
\end{equation*}
%		for all $u\in \mathbb R_+^n$ with $u_i = 0$.
\item\label{A5} ((Super)-quadratic Growth) There exist $\varepsilon>0$ sufficiently small and $K>0$ such that
\begin{equation*}
		|f_i(u)| \leq K(1+|u|^{2+\varepsilon}), \qquad \text{for all } i =1,\ldots, N, \quad u \in \mathbb R^N.
\end{equation*}
\end{enumerate}

The mass control assumption \eqref{A3} is a generalisation both to the condition of {\it mass conservation}
\begin{equation}\label{A3bis}\tag{A3'}
	\sum_{i=1}^N f_i(u) = 0
\end{equation}
and to the condition of {\it mass dissipation}
\begin{equation}\label{A3bisbis}\tag{A3''}
	\sum_{i=1}^{N} f_i(u) \leq 0.
\end{equation}

\begin{definition}\label{def.sol}
We call a function $u = (u_1, \ldots, u_N)$ a classical solution to \eqref{eq1} on $(0,T)$ if $\forall p>n$ we have $u_i \in C([0,T);L^{p}(\O)) \cap  C^{\infty}((0,T)\times \overline{\O})$ %\blue{To check if $u\in C([0,T];L^{p}(\Omega))$, $p>n$ or only $u\in C([0,T];L^p(\Omega))$ for all $p<\infty$?}
and $u$ satisfies each equation in \eqref{eq1} pointwise, cf. \cite{pazy}.
\end{definition}

The main result of this paper is the following theorem.
\begin{theorem}[Global existence of classical solutions] \label{thm:main}\hfill\\
Under Assumptions \eqref{A1}--\eqref{A5}, there exists a unique global classical solution to \eqref{eq1}. Moreover, 
\begin{itemize}
	\item[(i)] if $K_1 = 0$ in \eqref{A3}, then there exist a constant $L_0$ depending on $\O, n, N, d_i, K, K_0, \varepsilon$ and $\|u_{i0}\|_{L^{\infty}(\O)}$, and a constant $\xi >0$ such that
	\begin{equation*}
	\|u_i(t)\|_{L^\infty(\O)} \leq L_0(1+t^{\xi}) \qquad \text{ for all } \quad t\geq 0 \quad\text{ and for all }\quad i=1,\ldots, N.
	\end{equation*}
	
	\item[(ii)] if $K_0=0$ and $K_1 < 0$ in \eqref{A3}, the global solution decays to zero exponentially as $t\to \infty$, i.e. there exist $L_1 > 0$ and $\mu>0$ such that
	\begin{equation*}
		\|u_i(t)\|_{L^\infty(\Omega)} \leq L_1e^{-\mu t} \quad \text{ for all } \quad t\geq 0, \quad {and} \quad i=1,\ldots, N.
	\end{equation*}
\end{itemize}
\end{theorem}
\begin{remark}[Generalisations]\label{thm:gen}\hfill
\begin{itemize}[leftmargin=7mm]
	\item[(i)] The results of Theorem \ref{thm:main} can be directly generalised to the case where \eqref{A3} is replaced by the assumption: There exists $(\alpha_i)\in (0,\infty)^N$ such that
	\begin{equation}\label{genA4}
	\sum_{i=1}^{N}\alpha_i f_i(u) \leq K_0 + K_1\sum_{i=1}^Nu_i,\quad \text{ for all }\quad 
	u\in \R^N_+. 
	\end{equation}
	%			\item The assumption \eqref{A4} can be even relaxed to 
	%			\begin{equation*}
	%				\sum_{i=1}^{N}\alpha_i f_i(u) = \sum_{i=1}^{N}\beta_i u_i
	%			\end{equation*}
	%			for some $\alpha_i \in (0,\infty)$ and $\beta \in [0,\infty)$, by using the change of unknowns $w_i(x,t) = e^{\beta_i t}u_i(x,t)$. However, in this case Theorem \ref{thm:main} implies only global existence but not uniform-in-time bound. (\red{Seems NOT to be immediately true!})
%In space dimensions $n = 1,2$ and for quadratic  nonlinearities $f_i(u)$, it was shown in \cite{PSY17} that if $|\pa_jf_i(u)| \leq C(1+|u|)$ then \eqref{A4} can be weakened as
%In several applications, systems occur which dissipate (rather than conserve) mass, that is instead of \eqref{A3} to only assume
%	\begin{equation}\tag{A3bis}\label{A3bis}
%		\sum_{i=1}^{N}\alpha_if_i(u) \leq 0,\quad \text{ for all }\quad u \in \R^N. 
%	\end{equation}
%	This is partially answered in small dimensions, $n=1,2$, in \cite{PSY17}, or in recently higher dimensions in \cite{Sou18} assuming additionally the entropy dissipation
%	\begin{equation}\label{en_dis}
%		\sum_{i=1}^{N}f_i(u)\log u_i \leq 0, \quad \text{ for all } \quad u\in (0,\infty)^N.
%	\end{equation}
%However, extending the results of Theorem \ref{thm:main} in higher dimensions assuming only a mass dissipation structure of the form \eqref{A3bis} seems to be a challenging open problem (see the Remarks \ref{diss_mass_1} and \ref{diss_mass_2} below).\medskip
	
	\item[(ii)] (Boundary conditions) The results in Theorem \ref{thm:main} are proved for homogeneous Neumann boundary conditions, but the same results can also be obtained for homogeneous Dirichlet boundary conditions. Note that the proof of Theorem \ref{thm:main} doesn't apply directly to mixed boundary conditions. 
\end{itemize}
%			An interesting open question is whether the results of \eqref{thm:main} is still true if \eqref{A4} is replaced by the {\it dissipation of mass} condition
%			\begin{equation*}
%			\sum_{i=1}^{N}f_i(u) \leq 0.
%			\end{equation*}
\end{remark}

%\blue{Discussion about global existence for reaction-diffusion systems.}

The question of global existence for reaction-diffusion systems is a classical topic, yet still poses a lot of open and challenging problems. A main difficulty in studying global existence for reaction-diffusion systems are the lack of maximum principle estimates or invariant regions, i.e. a structural inapplicability of techniques from scalar equations. Let us mention some classical results from e.g. Rothe \cite{Rot84}, Amann \cite{Ama85}, Hollis, Martin and Pierre \cite{HMP87} or Morgan \cite{Mor89}. Most of these works assume some technical assumptions on the nonlinearities. Since the work of Pierre and Schmidt \cite{PS97}, more attention has been paid by mathematicians to study systems satisfying the only natural assumptions of mass control \eqref{A3} (or mass dissipation \eqref{A3bisbis}) and positivity preservation \eqref{A4}. 

It was already pointed out in \cite{PS97} that these two assumptions alone are not enough to prevent blow up (in $L^\infty$-norm), therefore some growth restrictions on the nonlinearities are necessary. In particular, the case of quadratic nonlinearities is of interest due to its relevance in many applications such as chemical reactions or Lotka-Volterra type systems (see Section \ref{sec:appl} for more details). Under \eqref{A1}--\eqref{A5}, \cite{DFPV07} showed the global existence of weak solutions. More generally, \cite{Pie03} showed the existence of global weak solutions as long as the nonlinearities belong to $L^1(Q_T)$. The global existence of classical (or strong) solutions to \eqref{eq1} usually requires additional assumptions to \eqref{A1}--\eqref{A5}. Such assumptions are, for instance, small space dimensions ($n=1,2$) (see \cite{GV10,PSY17,Tan17a}), or quasi-uniform diffusion coefficients (see \cite{CDF14,FLS16}), or the close-to-equilibrium regimes (see \cite{CC17,Tan17}). We refer the interested reader also to the excellent review \cite{Pie-Survey}. Interestingly, global existence of classical solution to \eqref{eq1} under Assumptions \eqref{A2}-\eqref{A3bis}-\eqref{A4}-\eqref{A5} in the case $\Omega = \R^n$ was already solved in \cite{Kan90}, but it went almost unnoticed until recently. 
Our main Theorem \ref{thm:main} extends the results of \cite{Kan90} to the case of mass control and bounded domains, and moreover gives polynomial bounds for the growth-in-time of the $L^\infty$-norm of solutions in case $K_1 = 0$, and exponential decay in case $K_0=0$ and $K_1 < 0$.

%A new proof of global classical solutions has been presented recently in \cite{CGV}, where the authors utilised the famous De-Giorgi method and developed some specific adaptations to reaction-diffusion systems. Additionally the our assumptions \eqref{A1}--\eqref{A5}, they assumed therein an entropy inequality condition of the form \eqref{en_dis}. 

Finally, we remark two recent related results of global classical solutions to nonlinear reaction-diffusion equations. 
Firstly, Caputo, Goudon and Vasseur \cite{CGV} used De Giorgi's methods to prove global 
classical solutions in case $\Omega = \mathbb R^n$ under assumptions
\eqref{A2}-\eqref{A3bis}-\eqref{A4}-\eqref{A5}, yet by additionally assuming 
the  \emph{entropy inequality} 
\begin{equation}\label{en_dis}
\sum_{i=1}^{N}f_i(u)\log u_i \leq 0.
\end{equation}
Although this entropic structure  appears naturally in many applications, it is not always satisfied, for example, in the case of skew-symmetric Lotka-Volterra systems (see Section \ref{sec:appl}).
Secondly, and more recently, Souplet presented in \cite{Sou18} the existence of global smooth solutions for 
at most quadratic nonlinear systems under the mass dissipation \eqref{A3bisbis} instead of \eqref{A3bis}, but still needs the entropy dissipation \eqref{en_dis}, yet with a simplified proof. 
\medskip

\medskip
Let us now describe the main ideas in proving Theorem \ref{thm:main}. The proof of Theorem~\ref{thm:main} crucially utilises the following lemma.
\begin{lemma}[Regularity Interpolation] \label{lem:main}\hfill\\
For some constant $d>0$, let $u$ be the solution to the inhomogeneous linear heat equation
\begin{equation}\label{heat-eq}
\begin{cases}
u_t - d\Delta u = \phi(x,t), &\quad(x,t)\in Q_T,\\
\nx u \cdot \nu = 0, &\quad(x,t)\in \pa\O\times (0,T)\\
u(x,0) = u_0(x), &\quad x\in\O.
\end{cases}
\end{equation}
Assume that
\begin{itemize}[topsep=5pt, leftmargin=10mm]
\item[(i)] there exists a H\"older exponent $\gamma \in [0,1)$ such that for all $x, x' \in \Omega$, and all $t\in (0,T)$, 
\begin{equation}\label{AA1}
|u(x,t) - u(x',t)| \leq H|x - x'|^{\gamma},
\end{equation}
	\item[(ii)] the inhomogeneity satisfies 
\begin{equation}\label{AA2}
\sup_{Q_T}|\phi(x,t)| \leq F.
\end{equation}
\end{itemize}
Then, the following uniform gradient estimate follows:
\begin{equation*}
	|\nx u(x,t)| \leq C\|u_0\|_{C^1(\Omega)} + BH^{\frac{1}{2-\gamma}}F^{\frac{1-\gamma}{2-\gamma}}, \qquad \text{for all }\quad (x,t)\in Q_T,
		\end{equation*}
where $B>0$ and $C>0$ are constants depending only on $\O$, $n$, $d$ and $\gamma$.
\end{lemma}

At the first glance, Lemma~\ref{lem:main} seems little helpful for the problem of global classical solutions to systems \eqref{eq1}
under Assumptions \eqref{A1}--\eqref{A5}. In fact, a crucial difficulty in showing 
global classical solutions  to \eqref{eq1} is the failure of comparison principles (except in some special cases) and thus the lack of sufficiently strong 
a-priori estimates. However, maximum principle and thus $L^{\infty}$ a-priori estimates are well known for the scalar equation \eqref{heat-eq}. 
So how could Lemma~\ref{lem:main} contribute to showing global classical solutions to systems \eqref{eq1}? 
\medskip

In this work, we consider first the case of "equality" and $K_1 = 0$ in \eqref{A3}, i.e.
\begin{equation}\label{A3eq}\tag{A3eq}
\sum_{i=1}^Nf_i(u)  = K_0.
\end{equation}
Under this assumption, we follow closely ideas of Kanel  \cite{Kan90}, who considered the case $\O = \R^n$, in which Lemma \ref{lem:main} plays an essential role. We extend his approach to bounded domain with smooth boundary. In particular, our approach requires careful H\"older regularity estimates up to the boundary for some parabolic equation of non-divergence form in bounded domains (see Lemma \ref{lem:uhat}), since the De Giorgi's technique usually gives only local H\"older continuity. As another refinement to \cite{Kan90}, our proof keeps track of the involved constants and shows that the $L^\infty$-norm of the solution grows at most polynomially in time. 

Such algebraic growth estimates are very useful in interpolation arguments and allow us, for instance, to prove exponential decay to zero of solutions in case $K_0 = 0$ and $K_1 < 0$. Another example for useful interpolation arguments with algebraic a-priori estimates are systems, which feature exponential $L^1$-convergence to equilibrium. This is true, for instance for complex balanced chemical reaction networks, see e.g. \cite{FT17,FT17a,DFT17}.
Indeed, for such systems, interpolation of slowly (i.e. algebraically) growing a-priori $L^\infty$-estimates with exponential $L^1$-convergence to equilibrium yields equilibration estimates in any $L^p$-norm for $p\in [1,\infty)$ as well as uniform-in-time bounds, see e.g. \cite{DF08,FLT}.
\medskip

After proving Theorem~\ref{thm:main} in the case of 
\eqref{A3eq}, the general case of \eqref{A3} follows from 
suitable transformations, see {\bf \textit{Step 2}} below.

%, our key idea is to transform the system into a new one satisfying \eqref{A3eq} by firstly changing the unknowns and secondly adding a suitable equation to the considered system. 

\medskip
In the following, we provide a detailed outline of the 
proof of the main Theorem~\ref{thm:main}.
In particular, we sketch how to manipulate 
system \eqref{eq1} under Assumptions 
\eqref{A1}--\eqref{A5} such that Lemma~\ref{lem:main} 
can be applied.

\subsection{Outline of the proof of Theorem~\ref{thm:main}}
The proof %of Theorem \ref{thm:main} 
is divided in two main steps:

\noindent{\bf \textit{Step 1: The case of condition \eqref{A3eq}.}} We consider first the case in which the nonlinearities satisfy \eqref{A3eq}.
\begin{enumerate}[topsep=5pt, leftmargin=7mm]

\item The existence of local in time, classical, non-negative solutions $u= (u_1,\ldots, u_N)$ on some time interval $[0,T)$ follows from classical references thanks to Assumptions \eqref{A1}, \eqref{A2} and \eqref{A4}, see e.g. \cite{pazy,Rot84,Pao} and Proposition \ref{thm:local} below. The same references also 
imply global existence of classical solutions provided 
that for all $T>0$
\begin{equation*}
\sup_{t\in (0,T)}\|u_i(\cdot, t)\|_{L^{\infty}(\O)} < +\infty, \quad \text{ for all } \quad i = 1,\ldots, N,
\end{equation*}
which is our aim in the following. 

\item For a constant $d> \max_{i=1,\ldots,N}\{d_i\}$, let $v_i$ be the solution to
\begin{equation}\label{vi}
			\begin{cases}
				\pa_t v_i - d\Delta v_i = u_i, &\quad(x,t)\in Q_T,\\
				\nx v_i \cdot \nu = 0, &\quad(x,t)\in \pa\O\times (0,T),\\
				v_i(x,0) = 0, &\quad x\in\O.
			\end{cases}
\end{equation}
%\begin{lemma}\label{lem:vd_sol}
%		For each $i=1,\ldots, N$, there exists a unique non-negative classical solution $v_i$ on $(0,T)$ to \eqref{vi}.
%	\end{lemma}	
%	\begin{proof}
Since $0\leq u_i \in C^\infty((0,T)\times \overline{\O})\cap C([0,T);L^p(\Omega))$ for some $p>n$, the existence, uniqueness and nonnegativity of $v_i$ follows from classical results.
Note that the solutions $v_i$ can be seen as mollified 
versions of the $u_i$. 
Moreover, 
by taking the sum over all $i=1,\ldots,N$ of \eqref{vi}, 
we calculate
%	\begin{lemma}\label{lem:zvd}
%		We have
%		\begin{equation*}
%			\sum_{i=1}^{N}u_i = z + \Delta v_d
%		\end{equation*}
%		where
%		\begin{equation}\label{zvd}
%			z = \sum_{i=1}^{N}(\pa_t v_i - d_i\Delta v_i) \quad \text{ and } \quad v_d = \sum_{i=1}^N(d-d_i)v_i.
%		\end{equation}
%	\end{lemma}
%		From \eqref{vi} we have
\begin{equation}\label{zvd}
%\sum_{i=1}^{N}u_i = \sum_{i=1}^{N}(\pa_tv_i - d\Delta v_i) = 
\underbrace{\sum_{i=1}^{N}(\pa_tv_i - d_i\Delta v_i)}_{\displaystyle=:z} - \Delta \underbrace{\sum_{i=1}^{N}(d-d_i) v_i}_{\displaystyle=: v_d} = \sum_{i=1}^{N}u_i. 
\end{equation}
%By setting $z = \sum_{i=1}^{N}(\pa_tv_i - d_i\Delta v_i)$ and $v_d = \sum_{i=1}^{N}(d - d_i)v_i$ we have (see Lemma \ref{lem:zvd})
%		\begin{equation}\label{eq2}
%			\sum_{i=1}^{N}u_i = z + \Delta v_d.
%		\end{equation}
\item Concerning the function $z$, direct calculations (see Lemma \ref{lem:boundz}) using assumption \eqref{A3eq} and \eqref{vi} show that $z$ satisfies the linear equation 
\begin{equation}\label{sysz}
\begin{cases}
\pa_t z - d\Delta z = K_0, &\qquad (x,t)\in Q_T,\\
\nx z \cdot \nu = 0, &\qquad (x,t)\in \pa\O\times (0,T),\\
z(x,0) = \sum_{i=1}^{N}u_{i0}(x), &\qquad  x\in\O,
\end{cases}
\end{equation}
and therefore the following comparison principle estimate holds
\begin{equation}\label{boundz}
\sup_{Q_T}|z| \leq \sum_{i=1}^{N}\|u_{i0}\|_{L^\infty(\Omega)} + K_0T.
\end{equation}
\item Secondly, it follows from the definition of $v_d$ in \eqref{zvd} that (see Lemma \ref{lem:vd}) 
\begin{equation}\label{vd}
v_d(x,t) = d \underbrace{\int_0^tz(x,s)ds}_{\displaystyle =:\wh{z}} - \underbrace{\int_0^t\sum_{i=1}^{N}d_iu_i(x,s)ds}_{\displaystyle =:\wh{u}},
\end{equation}
where the function $\wh{z}$ solves the following heat equation (see Lemma \ref{lem:zhat})
\begin{equation*}
\begin{cases}
\pa_t\wh{z} - d\Delta\wh{z} = z(0) + K_0t, \\
\nx \wh{z} \cdot \nu = 0, \qquad \wh{z}(0) = 0
\end{cases}
\end{equation*}
and is therefore H\"older continuous. 
\item Next follows the crucial observation of the H\"older continuity of 
the function $\wh{u}$ defined in \eqref{vd}, which solves (see Lemma \ref{lem:uhat_eq}) 
\begin{equation*}
\begin{cases}
b\,\pa_t\wh{u} - \Delta \wh{u} = \sum_{i=1}^{N}u_{i0} + K_0t, \\\nx \wh{u}\cdot \nu = 0, \qquad \wh{u}(0) = 0,
\end{cases}
\end{equation*}
and where the coefficient in front of the time derivative
\begin{equation*}
\qquad b(x,t) := \frac{\sum_{i=1}^{N}u_i(x,t)}{\sum_{i=1}^{N}d_iu_i(x,t)} 
\end{equation*}
is uniformly bounded due to its definition and the non-negativity of the $u_i$:
\begin{equation*}
\frac{1}{\max \{d_i\}} \leq b(x,t) \leq \frac{1}{\min \{d_i\}},
\qquad \text{for all }\quad(x,t)\in Q_T.
\end{equation*}
The boundedness of $b$ permits us to apply 
a classical result to prove H\"older continuity of $\wh{u}$,
see Lemma~\ref{lem:uhat} and its proof in the Appendix.
%By a technical lemma, it holds that $\wh{u}$ is  continuous. 
Therefore, by \eqref{vd} it follows (see Lemma \ref{lem:vdreg}) that $v_d$ is H\"older continuous with an exponent $\delta \in (0,1)$.

\item The two final steps prove boundedness 
of 
\begin{equation}\label{U}
|U| := \sup_{Q_T}\max_{i=1,\ldots,N}|u_i|
\end{equation}
by estimating the right hand side of \eqref{zvd} in terms of regularity estimates of the left hand side terms. 
This strategy requires, in particular, to control  
$\sup_{Q_T}|\Delta v_d|$ in terms of $|U|$.
First, we observe 
that $v_d$ satisfies (see Lemma \ref{lem:Lapvd})
\begin{equation}\label{vdeq}
\begin{cases}
\pa_tv_d - d\Delta v_d = \sum_{i=1}^{N}(d-d_i)u_i=: u_d,\\
\nx v_d \cdot \nu = 0,  \qquad v_d(0) = 0,
\end{cases}
\end{equation}
which will allow us to apply Lemma \ref{lem:main}.
At this point, when contemplating why the presented approach is able to succeed at the end of the day, 
we recall that the lack of maximum principle estimates 
for \eqref{eq1} stems from the different diffusion coefficients 
$d_i$. Note then, that $u_d$ constitutes the difference between a $d$-weighted sum (or equally average) of the 
$u_i$ and a $d_i$-weighted sum (or average). 
Accordingly, the function $v_d$ has the same interpretation 
in terms of the $v_i$, i.e. mollified versions of the $u_i$.
The key observation is that the function $v_d$
satisfies with \eqref{vdeq} a nice parabolic equation (rather than a parabolic system). 
Consequentially, the following estimates on $u_d$ and $v_d$ show that 
both those difference between $d$-weighted and $d_i$-weighted sums satisfy the required regularity estimates in terms of $|U|$ (and without a closeness condition between $d$ and the $d_i$ as exploited in \cite{CDF14}).

More precisely, Lemma \ref{lem:main} implies for $v_d$ as a $\delta$-H\"older continuous solution of \eqref{vdeq} that
(see Lemma~\ref{lem:gradvd})
\begin{equation}\label{boundvd}
|\nx v_d| \leq C_T\sup_{Q_T}|u_d|^{\frac{1-\delta}{2-\delta}} \leq C_T|U|^{\frac{1-\delta}{2-\delta}},
%\qquad \text{where}\quad |U| := \sup_{Q_T}\max_{i=1,\ldots,N}|u_i|,
\end{equation}
where the constant $C_T$ depends at most polynomially on $T$ for any existence interval $[0,T)$.

Analog, from the equations $\pa_tu_i - d_i\Delta u_i = f_i(u)$ and assumption \eqref{A5}, Lemma \ref{lem:main} with $\gamma = 0$ yields (see Lemma~\ref{lem:gradud})
\begin{equation}\label{boundud}
\begin{aligned}
|\nx u_d| &\leq C\sum_{i=1}^{N}|\nx u_i| \leq C_T(1+|U|^{1/2}\sup_{Q_T}\max_{i=1,\ldots,N}|f_i(u)|^{1/2})\\
&\leq C_T\left(1+ |U|^{1/2}\left(1+|U|^{1+\frac \varepsilon 2}\right)\right)
\leq C_T\left(1+|U|^{\frac{3+\varepsilon}{2}} \right).
\end{aligned}
\end{equation}
%Hence, after differentiating \eqref{vdeq} w.r.t. $x_i$ and by using once more Lemma \ref{lem:main} with $\gamma = 0$, 
%we estimate with \eqref{boundvd} and \eqref{boundud}
{\color{black}
To estimate $\Delta v_d$, we use a second order estimate of the heat semigroup $e^{t\Delta}$ to obtain\footnote{This idea was used in \cite{Sou18}.} (see Lemma \ref{lem:Lapvd})
\begin{equation}\label{boundDvd}
\begin{aligned}
\sup_{Q_T}|\Delta v_d| &\leq C_T\sup_{Q_T}(|v_d|+|\nx v_d|)^{1/2}\sup_{Q_T}(|u_d| + |\nx u_d|)^{1/2}
\\
&\leq C_T\left(1+|U|^{\frac{1-\delta}{2-\delta}}\right)^{1/2}\left(1+|U|^{\frac{3+\varepsilon}{2}}\right)^{1/2}\\
&\leq C_T\left(1+|U|^{\frac 34 + \frac{\varepsilon}{4} + \frac{1-\delta}{2(2-\delta)} }\right).
\end{aligned}
\end{equation}
Note that in the case $\Omega = \mathbb R^d$, one can obtain \eqref{boundDvd} by differentiating \eqref{vdeq} and apply directly Lemma \ref{lem:main} to the equation of $\pa_{x_j}v_d$ with $\gamma = 0$.
}
\item Finally, from \eqref{zvd}, \eqref{boundz} and \eqref{boundDvd}, it follows
\begin{equation*}
|U| \leq C_T\left(1+|U|^{\frac 34 + \frac{\varepsilon}{4} + \frac{1-\delta}{2(2-\delta)} }\right)
\end{equation*}
and therefore, since $\frac 34 + \frac{\varepsilon}{4} + \frac{1-\delta}{2(2-\delta)}<1$ for $\varepsilon$ sufficiently small, 
\begin{equation*}
	|U| \leq C_T, \qquad \text{for all}\quad (x,t)\in Q_T,
\end{equation*}
which implies that the solutions $u_i$ can be extended 
globally as bounded and therefore smooth solutions, 
see e.g. \cite{pazy,Rot84}. 
\end{enumerate}	

\medskip

One important observation here is that the results of \textit{\textbf{Step 1}} are still true, except for the fact that the $L^\infty$-norm in this case might grow faster than polynomial, if we replace the constant $K_0$ by a function $K_0(t)$ which is continuous on $[0,\infty)$. This fact will be helpful in \textit{\textbf{Step 2.}}

	\medskip

\noindent{\bf \textit{Step 2: The case of condition \eqref{A3}.}} The general assumption \eqref{A3} can be transformed into condition \eqref{A3eq} (where $K_0$ is replace by a continuous function $K_0(t)$) by the help of 
rescaling and the addition of an appropriate equation to the system:
\begin{enumerate}
\item By defining
\begin{equation*}
	w_i(x,t) = e^{-K_1t}u_i(x,t)  \quad \text{ or equivalently } \quad u_i(x,t) = e^{K_1 t}w_i(x,t)
\end{equation*}
we obtain (see Section \ref{sec:control}) the following system for $w = (w_1, \ldots, w_N)$
\begin{equation}\label{eq_w}
	\begin{cases}
		\pa_t w_i - d_i\Delta w_i = g_i(w), &(x,t)\in Q_T \quad \text{ for } \quad i = 1,\ldots, N,\\
		\nx w_i \cdot \nu = 0, &(x,t)\in \pa\Omega\times (0,T),\\
		w_i(x,0) = u_{i0}(x), &x\in\Omega
	\end{cases}
\end{equation}
where the nonlinearities $g_i(w) = e^{-K_1t}(f_i(u(w)) - K_1e^{K_1t}w_i)$.
It is obvious that nonlinearities $g_i(w)$ satisfy the assumptions \eqref{A4}--\eqref{A5}, while \eqref{A3} is changed to
\begin{equation}\label{sum_g}
	\begin{aligned}
	\sum_{i=1}^N g_i(w) \leq K_0e^{-K_1 t}.
	\end{aligned}
\end{equation}

\item To obtain \eqref{A3eq} from \eqref{sum_g}, we introduce an $(N+1)$-th equation for $w_{N+1}$ as
\begin{equation*}\qquad \quad
	\partial_t w_{N+1} - \Delta w_{N+1} = g_{N+1}(w):= K_0e^{-K_1 t} - \sum_{i=1}^{N}g_i(w),
\end{equation*}
together with boundary condition $\nx w_{N+1}\cdot \nu = 0$ and initial data $w_{N+1}(x,0) = 0$. It is immediate that $g_{N+1}(w)$ satisfies the conditions \eqref{A4} and \eqref{A5}. We then obtain a new, enlarged system for $\wt{w} = (w_1, \ldots, w_{N+1})$
\begin{equation}\label{eq_w_new}
	\begin{cases}
		\pa_t w_i - d_i\Delta w_i = g_i(\wt w):= g_i(w), &(x,t)\in Q_T \quad \text{ for } \quad i = 1,\ldots, N+1,\\
		\nx w_i \cdot \nu = 0, &(x,t)\in \pa\Omega\times (0,T), \quad i=1,\ldots, N+1,\\
		w_i(x,0) = u_{i0}(x), &x\in\Omega, \quad i=1,\ldots, N,\\
		w_{N+1}(x,0) = 0, &x\in\Omega
	\end{cases}
\end{equation}
in which the nonlinearities now satisfy
\begin{equation*}
	\sum_{i=1}^{N+1}g_i(\wt w) = K_0e^{-K_1 t}
\end{equation*}
and the problem is thus reduced to {\bf \textit{Step 1.}}, taking into account that last remark therein that $K_0$ could be a continuous function of $t$. Note that if $K_0 = 0$ and $K_1 < 0$ then $\wt w$ has at most polynomial growth in time. Thus, from $u_i(x,t) = e^{K_1 t}w_i(x,t)$ we get for $K_1 < 0$ that $u_i$ decays exponentially to zero.
\end{enumerate}

\medskip
%	The ideas follow closely from \cite{Kan90}, where the case $\O = \R^d$ was considered, with suitable changes applying to the case of bounded domain with smooth boundary. Especially, for proving the H\"older continuity of $\wh{u}$ in bounded domain one needs to carefully examines it near the boundary, since the De-Giorgi's technique usually gives only local H\"older continuity. Moreover, we keep track of the involved constants and therefore show that the $L^\infty$-norm of the solution grows at most polynomially in time. This fact turns out to be useful, for example, in the case of reversible reactions, where the interpolation with exponential $L^1$-convergence to equilibrium leads to equilibration in any $L^p$-norm for $p\in [1,\infty)$. We also remark that the results in this paper assumes only mass conservation \eqref{A3}, and \textit{not the entropy inequality}
%	\begin{equation*}
%		\sum_{i=1}^{n}f_i(u)\log u_i \leq 0
%	\end{equation*}
%	as considered in \cite{CGV}. This assumption, though appears naturally in many applications, is not always satisfied, for example, in the case of skew-symmetric Lotka-Volterra system (see Section \ref{sec:appl}).

A direct \textbf{application} of Theorem~\ref{thm:main} allows to prove global classical solutions in all space dimensions 
to nonlinear, skew-symmetric Lotka-Volterra systems with diffusion of the form
	\begin{equation}\label{LV}
		\begin{cases}
		\pa_t {u}_i -  d_i\Delta {u}_i = \left(-\tau_i + \sum_{j=1}^{N}a_{ij}{u}_j\right){u}_i =: f_i(u), &(x,t)\in Q_T,\\
		\nx {u}_i \cdot \nu = 0, &(x,t)\in \pa\Omega\times (0,T),\\
		u_i(x,0) = u_{i0}(x), &x\in\O,
		\end{cases}
	\end{equation}
	where $A = (a_{ij})\in \R^{N\times N}$ is skew-symmetric, i.e. $A^{\top} + A = 0$ and $\tau = (\tau_1, \ldots, \tau_N)\in \R^N$. Note that such Lotka-Volterra systems do not satisfy an entropy structure as required by \cite{CGV,Sou18}. We will show that \eqref{LV} has a unique global classical solution in all dimensions, and moreover, if $\tau_i > 0$ for all $i=1,\ldots, N$, then the solution decays to zero exponentially (see Section~\ref{sec:appl}). 
	
	\medskip
\noindent\textbf{The organisation of the paper} is as follows: In Section~\ref{sec:lem}, we prove the crucial Lemma~\ref{lem:main}. In Section \ref{sec:thm}, we first present the proof of Theorem~\ref{thm:main} with condition \eqref{A3eq}, while the proof of Theorem \ref{thm:main} for general condition \eqref{A3} is presented in Section~\ref{sec:control}. Section~\ref{sec:appl} is devoted to some applications of the main results. Finally, a technical proof of H\"older continuity for $\wh{u}$ is presented in the Appendix~\ref{appendix}.
	
	\medskip
\noindent\textbf{Notations:}
\begin{itemize}[topsep=5pt, leftmargin=6mm]
		\item $\R_+ = [0,\infty)$, $\R_+^N = [0,\infty)^N$.
		\item The usual norm of $L^p(\O)$ is denoted by $\|\cdot\|_{L^p(\O)}$ for any $1\leq p \leq \infty$.
		\item For $T>0$, we denote by $Q_T = \O\times (0,T)$ and for any $1\leq p < +\infty$
		\[
			\|f\|_{L^p(Q_T)} = \left[\int_{Q_T}|f|^pdxdt\right]^{1/p}, \quad \|f\|_{L^{\infty}(Q_T)} = {\mathrm{ess}\sup}_{Q_T}|f|.
		\]
		\item The generic constant $C_i$, $i=1,\ldots, 12$ depends only on the data $\O$, $N$, $n$, $d_i$, $K$, $K_0$, $K_1$, $\varepsilon$ and $\|u_{i0}\|_{L^\infty(\O)}$. In particular, $C_i$ {\it does not depend on } $T>0$.
	\end{itemize}
	
	\section{Proof of key Lemma \ref{lem:main}}\label{sec:lem}
	In case $\O$ is a bounded domain, we denote by $G(x,t)$ the Green function of the heat equation $\pa_t u - d\Delta u = 0$ subject to a homogeneous Neumann boundary condition. When $\O = \R^n$, let $G(x,t)$ be the fundamental solution, i.e. $G(x,t) = \frac{1}{(4d\pi t)^{n/2}}e^{-{|x|^2}/{4dt}}$. Let $k>0$ be a constant which 
serves as a kind of interpolation parameter, to be specified later. We rewrite \eqref{heat-eq} as
	\[
		\pa_tu - d\Delta u + ku = \phi(x,t) + ku,
	\]
	and use the representation formula to have
	\begin{equation}\label{rep}
		u(x,t) = e^{-kt}\wt{u}(x,t) + \int_0^te^{-k(t-s)}\int_{\O}G(x-y,t-s)[\phi(y,s) + ku(y,s)]dyds
	\end{equation}
	where $\wt{u}(x,t)$ is the solution to 
\begin{equation*}
\begin{cases}
\pa_t\wt{u} - d\Delta \wt{u} = 0, \\
\nx \wt{u}\cdot \nu = 0, \qquad \wt{u}(x,0) = u_0(x).
\end{cases}
\end{equation*} 
%Note that $\wt{u}$ is smooth and the maximum principle (applied to differentiated equations of $\wt{u}$) yields 
Using the property of heat semigroup, see e.g. \cite[Eq. (2.39)]{Mor83} we have
\begin{equation}\label{utilde}
		\sup_{t>0}\sup_{x\in\O}|\nx \wt{u}(x,t)| \leq C\|u_0\|_{C^1(\Omega)}.
\end{equation}
By differentiating \eqref{rep} in the spatial variables and using $\int_{\O}\nx G(x-y,t-s)u(x,s)dy = 0$ thanks to the Neumann boundary condition, we have
	\begin{equation}\label{unab}
		\nx u(x,t) = e^{-kt}\nx \wt{u}(x,t) + \int_0^te^{-k(t-s)}\int_{\O}\nx G(x-y,t-s)[\phi(y,s) + k(u(y,s) - u(x,s)]dyds.
	\end{equation}
	It follows from \eqref{unab} and \eqref{utilde} that
\begin{align}
%\begin{aligned}
|\nx u(x,t)| &\leq C\|u_0\|_{C^1(\Omega)} \nonumber \\
&\quad + \int_0^te^{-k(t-s)}\int_{\O}|\nx  G(x-y,t-s)|[|\phi(y,s)| + k|u(y,s) - u(x,s)|]dyds\nonumber \\
		&\leq C\|u_0\|_{C^1(\Omega)} + F\int_0^te^{-k(t-s)}\int_{\O}|\nx  G(x-y,t-s)|dyds\nonumber \\
		& \quad + kH\int_0^te^{-k(t-s)}\int_{\O}|\nx  G(x-y,t-s)||x-y|^{\gamma}dyds\label{unab1}
\end{align}
%\end{equation}
where we have used the assumptions \eqref{AA1} and \eqref{AA2} in the last step.

%\textcolor{blue}{Detail Case 2, recall case 1} 

	We now distinguish the two cases, when $\O$ is a bounded domain and $\O = \R^n$.
	\begin{description}[leftmargin=6 mm]
		\item[Case 1] Let $\O$ be a bounded domain with smooth boundary.
		In this case use the following point-wise gradient estimate on the Green function (see e.g. \cite{LY86,Wan97} or \cite{Dun04, Mor83}),
		\begin{equation*}
		|\nx G(x-y, t-s)| \leq c_n(t-s)^{-\frac{n+1}{2}}e^{-\kappa_n\frac{|x-y|^2}{t-s}}
		\end{equation*}
		where positive constants $c_n$ and $\kappa_n$ depend only on the dimension $n$, the domain $\Omega$ and the diffusion coefficients $d$. Using this bound, we can estimate further
		\begin{equation}\label{inter1}
		\begin{aligned}
		|\nx u(x,t)| \leq C\|u_0\|_{C^1(\Omega)} &+ c_nF\int_0^t\frac{e^{-k(t-s)}}{(t-s)^{\frac{n+1}{2}}}\int_{\O}e^{-\kappa_n\frac{|x-y|^2}{t-s}}dyds\\
		&+ c_nkH\int_0^t\frac{e^{-k(t-s)}}{(t-s)^{\frac{n+1}{2}}}\int_{\O}|x - y|^{\gamma}e^{-\kappa_n\frac{|x-y|^2}{t-s}}dyds.
		\end{aligned}
		\end{equation}
		We denote by $(\mathrm{I})$ and $(\mathrm{II})$ the second and last terms on the right hand side of \eqref{inter1}. 		
		By the change of variables $z = \sqrt{\frac{\kappa_n}{t-s}}(y-x)$, we have for any $\delta \geq 0$
			\begin{equation}\label{identity}
				\begin{aligned}
		\quad		\int_{\Omega}|x-y|^{\delta}e^{-\kappa_n\frac{|x - y|^2}{t-s}}dy &= \left(\frac{t-s}{\kappa_n}\right)^{\frac{n+\delta}{2}}\int_{\sqrt{\frac{\kappa_n}{t-s}}(\Omega - x)}|z|^{\delta}e^{-|z|^2}dz
				\le \left(\frac{t-s}{\kappa_n}\right)^{\frac{n+\delta}{2}}\int_{\R^n}|z|^{\delta}e^{-|z|^2}dz\\
				&\leq \omega_{n-1}\Gamma\left(\frac{n+\delta}{2}\right)\left(\frac{t-s}{\kappa_n}\right)^{\frac{n+\delta}{2}},
			\end{aligned}
		\end{equation}
		where $\omega_{n-1}$ is the surface area of the $n-1$-dimensional unit sphere and $\Gamma$ is the Gamma function. Therefore
		\begin{equation}\label{IIIest}
		(\mathrm{I}) \leq c_n\kappa_n^{-n/2}F\int_0^t\frac{e^{-k(t-s)}}{\sqrt{t-s}}ds\int_{\R^n}e^{-|z|^2}dz \leq c_n\kappa_n^{-n/2}F\,\Gamma(n/2)\frac{\sqrt \pi}{\sqrt{k}} =: \frac{B_1F}{\sqrt{k}}
		\end{equation}
		where 
		\[
		B_1 = c_n\kappa^{-n/2}\Gamma(n/2)\sqrt{\pi}.\]
		For $(\mathrm{II})$, we have
		\begin{equation}\label{IVest}
		(\mathrm{II}) \leq c_n\kappa_n^{-(n+\gamma)/2}kH\int_0^t\frac{e^{-k(t-s)}}{(t-s)^{\frac{1-\gamma}{2}}}ds\int_{\R^n}|z|^{\gamma}e^{-|z|^2}dz \leq B_2kH\int_0^t\frac{e^{-k(t-s)}}{(t-s)^{\frac{1-\gamma}{2}}}ds
		\end{equation}
		where
		\[
		 %\quad \text{and} \quad 
		B_2 = c_n\kappa_n^{-(n+\gamma)/2}\Gamma\Bigl(\frac{n+\gamma+1}{2}\Bigr).
		\]
		The last term on the right hand side of \eqref{IVest} is estimated using the change of variable $\tau = \sqrt{k(t-s)}$,
				\begin{equation}\label{abc}
					(\mathrm{II}) \leq B_2H(\sqrt{k})^{1-\gamma}\int_0^{\sqrt{kt}}e^{-\tau^2}\tau^{\gamma}d\tau \leq B_3H(\sqrt{k})^{1-\gamma}
				\end{equation}
		with
		\[
			B_3 = B_2\Gamma\Bigl(\frac{\gamma+1}{2}\Bigr).
		\]
		From \eqref{inter1}, \eqref{IIIest} and \eqref{abc}, we have
				\begin{equation*}
					|\nx u(x,t)| \leq C\|u_0\|_{C^1(\Omega)} + \frac{B_1F}{\sqrt{k}} + B_3H(\sqrt{k})^{1-\gamma}.
				\end{equation*}
				By choosing
				\begin{equation*}
					\sqrt{k} = \left[\frac{B_1F}{B_3H(1-\gamma)} \right]^{\frac{1}{2-\gamma}},
				\end{equation*}
				we obtain the desired estimate
				\begin{equation*}
					|\nx u(x,t)| \leq C\|u_0\|_{C^1(\Omega)} +  BF^{\frac{1-\gamma}{2-\gamma}}H^{\frac{1}{2-\gamma}}
				\end{equation*}
				with
				\[
					B = \left[(1-\gamma)^{\frac{1}{2-\gamma}} + (1-\gamma)^{\frac{\gamma-1}{2-\gamma}} \right]B_1^{\frac{1-\gamma}{2-\gamma}}B_3^{\frac{1}{2-\gamma}}.
				\]
		
		\medskip
		
		\item[Case 2] $\O = \R^n$. In this case, we use the explicit representation of the fundamental solution to obtain
		\begin{equation*}
			\nx G(x-y, t-s) = \frac{1}{(4\pi d(t-s))^{n/2}}\frac{y - x}{2d(t-s)}e^{-\frac{|x-y|^2}{4d(t-s)}}
		\end{equation*}
		and consequently get from \eqref{unab1} that 
		\begin{equation}\label{inter}
			\begin{aligned}			
			|\nx u(x,t)| &\leq C\|u_0\|_{C^1(\mathbb R^n)}+ F\frac{1}{2d(4\pi d)^{n/2}}\int_0^t\frac{e^{-k(t-s)}}{(t-s)^{n/2+1}}\int_{\R^n}|x-y|e^{-\frac{|x-y|^2}{4d(t-s)}}dyds\\
			&\quad + kH\frac{1}{2d(4\pi d)^{n/2}}\int_0^t\frac{e^{-k(t-s)}}{(t-s)^{n/2+1}}\int_{\R^n}|x-y|^{1+\gamma}e^{-\frac{|x-y|^2}{4d(t-s)}}dyds.
			\end{aligned}			
		\end{equation}
		Denote by $(\mathrm{III})$ and $(\mathrm{IV})$ the second and the last term on the right hand side of \eqref{inter}, respectively. We estimate these two terms similarly as in {\bf Case 1} to have
		\begin{equation*}
			(\mathrm{III}) \leq \frac{B_4F}{\sqrt{k}}, \qquad \text{where} \quad B_4 = \frac{\omega_{n-1}}{\pi^{(n-1)/2}\sqrt{d}}\,\Gamma\Bigl(\frac{n+1}{2}\Bigr)
		\end{equation*}
		and
		\begin{equation*}
			(\mathrm{IV}) \leq B_5H(\sqrt{k})^{1-\gamma},\qquad \text{where} \quad B_5 =  \frac{\omega_{n-1}}{\pi^{n/2}}\,2^{\gamma-1} d^{\frac{\gamma - 1}{2}}\,\Gamma\Bigl(\frac{1+\gamma}{2}\Bigr)\Gamma\Bigl(\frac{n+1+\gamma}{2} \Bigr).
		\end{equation*} 
		Hence
		\begin{equation*}
			|\nx u(x,t)| \leq C\|u_0\|_{C^1(\mathbb R^n)} + \frac{B_4F}{\sqrt{k}} + B_5H(\sqrt{k})^{1-\gamma},
		\end{equation*}
		and therefore
		\begin{equation*}
			|\nx u(x,t)| \leq C\|u_0\|_{C^1(\mathbb R^n)} + BF^{\frac{1-\gamma}{2-\gamma}}H^{\frac{1}{2-\gamma}}
		\end{equation*}
		with
		\begin{equation*}
			B = \left[(1-\gamma)^{\frac{1}{2-\gamma}} + (1-\gamma)^{\frac{\gamma-1}{2-\gamma}} \right]B_4^{\frac{1-\gamma}{2-\gamma}}B_5^{\frac{1}{2-\gamma}}.
		\end{equation*}
	\end{description}
	\section{Proof of Theorem \ref{thm:main} with condition \eqref{A3eq}}\label{sec:thm}
	\begin{proposition}[Local existence]\label{thm:local}\hfill\\
		Assume \eqref{A1}, \eqref{A2} and \eqref{A4}. Then, there exists an interval $[0,T)$ (which can be chosen maximal) and a corresponding unique nonnegative classical solution $u$ to \eqref{eq1} on $(0,T)$. Moreover, in order to extend the solution globally, it is sufficient to prove for all $T>0$
\begin{equation}\label{bucri}
\lim_{t\uparrow T}\|u_i(t)\|_{L^{\infty}(\O)} < +\infty \qquad \forall i=1,\ldots, N, \qquad \Rightarrow \qquad T = +\infty.
\end{equation}
\end{proposition}
	\begin{proof}
		Since $u_{i0}\in L^\infty(\O)$ and the nonlinearities are locally Lipschitz, the local existence follows from classical results (see e.g. \cite{pazy,Rot84,Ama85}). The quasi-positivity of the nonlinearity \eqref{A4} implies propagation of non-negativity of initial data, e.g. \cite{Pao,Pie-Survey}.
	\end{proof}
\begin{remark}[Weaker blow-up criteria]\hfill\\
The continuation criteria \eqref{bucri} is standard for nonlinear reaction-diffusion systems. For the systems considered here,
it can be weakened. For instance, with assumptions \eqref{A3eq} and \eqref{A4}, by applying improved duality estimates as in \cite{CDF14}, global classical solutions can be shown from the existence of an exponent $p > (1+\varepsilon)(1+n/2)$ such that
\begin{equation*}
	\limsup_{t\uparrow T}\|u_i\|_{L^{p}(Q_t)} < +\infty \quad \text{ for all } \quad i = 1,\ldots, N,\qquad \Rightarrow \qquad T = +\infty.
\end{equation*}
	\end{remark}
\begin{remark}[Smooth initial data]\label{re_smooth}\hfill\\
Thanks to Definition \ref{def.sol} and the smoothing effect of the heat semigroup, it is well-known that the local solutions of Proposition \ref{thm:local} are smooth for positive times, i.e. that for all $0<t_0 <T$ holds $u_i(t_0) \in C^\infty(\overline{\O})$ for all $i=1,\ldots, N$. This allows us to shift the initial time to $t_0$ and therefore consider system \eqref{eq1} with smooth initial data in $C^{\infty}(\overline\O)$. Note that all parabolic compatibility conditions are satisfy at $t_0>0$. 
Consequentially, from now onwards, we weaken (w.l.o.g.) Assumption \eqref{A2} by only considering smooth
initial data $u_{i0}\in C^{\infty}(\overline{\O})$ satisfying a homogeneous Neumann boundary condition for all $i=1,\ldots, N$. Consequently, we have $u_i \in C^{\infty}([0,T)\times \overline{\O})$.
\end{remark}

%\setcounter{\theenumi}{6}
%\addtocounter{\theenumi}{6}
%\stepcounter{\theenumi}
From now on, because of Remark \ref{re_smooth},  we assume instead of \eqref{A2} that the initial data satisfy 
\begin{enumerate}[label=(A\theenumi'),ref=A\theenumi,leftmargin=11mm]
	\setcounter{enumi}{1}
\item\label{A6} (Smooth Nonnegative Initial Data) For all $i=1,\ldots N$: $0\le u_{i0}\in C^{\infty}(\overline{\O})$.
\end{enumerate}
Moreover, we assume %for convenience additionally
\begin{enumerate}[label=(A\theenumi),ref=A\theenumi,leftmargin=10mm]
	\setcounter{enumi}{5}
\item\label{A7} (Diffusion Coefficient $d$ in \eqref{vi}) Throughout this section, we fix the diffusion coefficient $d>\max_{i=1,\ldots,N}\{d_i\}$. Moreover, let $v_i$ be the unique solution to \eqref{vi} on $Q_T$.
\end{enumerate}
	
\begin{lemma}\label{lem:boundz}
		Let $z$ be defined in \eqref{zvd}. Then, $z$ is the solution to \eqref{sysz} and 
		\begin{equation*}
			\sup_{Q_T}|z| \leq M + K_0T:= \sum_{i=1}^{N}\|u_{i0}\|_{L^{\infty}(\O)} + K_0T.
		\end{equation*}
	\end{lemma}
\begin{proof}
First, we show that $z$ solves \eqref{sysz} and 
%the heat equation
%		\begin{equation*}
%			\pa_tz - d\Delta z = 0, \quad \nx z \cdot \nu = 0, \quad z(x,0) = \sum_{i=1}^{N}u_{i0}(x)
%		\end{equation*}
the rest follows from the comparison principle. Indeed, by using definition \eqref{zvd}, we calculate
		\begin{equation*}
			\begin{aligned}
				\pa_tz - d\Delta z &= \pa_t\biggl(\,\sum_{i=1}^N(\pa_t v_i - d_i\Delta v_i)\biggr) - d\Delta \biggl(\, \sum_{i=1}^N(\pa_t v_i - d_i\Delta v_i)\biggr)\\
				&= \sum_{i=1}^{N}\left[\pa_t\left(\pa_t v_i - d\Delta v_i\right) - d_i\Delta \left(\pa_t v_i - d\Delta v_i\right) \right]
= \sum_{i=1}^{N}(\pa_t u_i - d_i\Delta u_i) \qquad (\text{using (\ref{vi})})\\
				&= \sum_{i=1}^{N}f_i(u)= K_0 \qquad (\text{using (\ref{A3eq})}).
			\end{aligned}
		\end{equation*}
Moreover, on $\partial\Omega$, we have $\nx u_i \cdot \nu = \nx v_i \cdot \nu = 0$ and $\pa_tv_i - d\Delta v_i = u_i$ implying $\nx(\Delta v_i)\cdot \nu = 0$. Therefore, the boundary condition $\nx z \cdot \nu = 0$ follows from the definition $z = \sum_{i=1}^{N}(\pa_t v_i - d_i\Delta v_i)$. For initial data, since $v_i(x,0) = 0$, we have
\begin{equation*}
			z(x,0) = \sum_{i=1}^{N}(\pa_tv_i(x,0) - d_i\Delta v_i(x,0)) = \sum_{i=1}^{N}(\pa_tv_i(x,0) - d\Delta v_i(x,0)) = \sum_{i=1}^{N}u_{i}(x,0) = \sum_{i=1}^{N}u_{i0}(x).
\end{equation*}
\end{proof}
%\begin{remark}\label{diss_mass_1}
%	If the conservation of mass $\sum_{i=1}^{N}f_i(u) = 0$ is replaced by the dissipation of mass $\sum_{i=1}^{N}f_i(u) \leq 0$, then it is not clear how to get the uniform bound of $z$ in Lemma \ref{lem:boundz} since one does not know the sign of $z$.
%\end{remark}

\begin{lemma}\label{lem:vd}
The function $v_d$ defined in \eqref{zvd} satisfies
	\begin{equation*}
		v_d = d\wh{z} - \wh{u}
	\end{equation*}
with $\wh{z}$ and $\wh{u}$ as defined in \eqref{vd}.
%	and therefore $v_d$ is H\"older continuous in $Q_T$, i.e. there exist $C_5>0$ and $\delta \in (0,1)$ such that
%	\begin{equation}\label{vd_Holder}
%		|v_d(x,t) - v_d(x', t')| \leq C_5T(|x -x'|^\delta + |t-t'|^{\delta/2}) \quad \text{ for all } \quad (x,t), (x', t') \in Q_T.
%	\end{equation}
\end{lemma}
\begin{proof}
Integrating \eqref{vi} over $(0,t)$ and using $v_i(x,0) = 0$ yields
\begin{equation*}
	v_i(x,t) - d\Delta \int_0^tv_i(x,s)ds = \int_0^tu_i(x,s)ds.
\end{equation*}
Hence
\begin{align*}
%	\begin{aligned}
			v_d(x,t) &=\sum_{i=1}^{N}(d-d_i)v_i(x,t)
			= d\int_0^t\sum_{i=1}^N(d-d_i)\Delta v_i(x,s)ds + \int_0^t\sum_{i=1}^{N}(d-d_i)u_i(x,s)ds\\
			&= d\int_0^t\left[d\Delta \sum_{i=1}^Nv_i(x,s) + z(x,s) - \sum_{i=1}^{N}\pa_sv_i(x,s)\right]ds + \int_0^t\sum_{i=1}^N(d-d_i)u_i(x,s)ds\\
			&= d\int_0^t\left[z(x,s) - \sum_{i=1}^Nu_i(x,s)\right]ds + \int_0^t\sum_{i=1}^N(d-d_i)u_i(x,s)ds\\
			&= d\int_0^tz(x,s)ds - \int_0^t\sum_{i=1}^{N}d_iu_i(x,s)ds
			= d\wh{z} - \wh{u}.
%		\end{aligned}
\end{align*}
%	From Lemmas \ref{lem:zhat} and \ref{lem:uhat} we obtain first for $|x -x'|, |t - t'| \leq 1$ that
%	\begin{equation}\label{est_vd1}
%		\begin{aligned}
%			|v_d(x,t) - v_d(x',t')| &\leq d|\wh{z}(x,t) - \wh{z}(x',t')| + |\wh{u}(x,t) - \wh{u}(x',t')|\\
%			&\leq \max\{dC_1,C_2 \}T(|x-x'|^{\delta} + |t-t'|^{\delta/2})
%		\end{aligned}
%	\end{equation}
%	where $\delta = \min\{\alpha, \beta\}$ with $\beta$ and $\alpha$ are in Lemmas \ref{lem:zhat} and \ref{lem:uhat}. When $|x - x'| \geq 1$ or $|t - t'| \geq 1$ we estimate
%	\begin{equation}\label{est_vd2}
%		\begin{aligned}
%		|v_d(x,t) - v_d(x',t')| &\leq 2\|v_d\|_{L^{\infty}(Q_T)}(|x - x'|^{\delta} + |t-t'|^{\delta/2})\\
%		&\leq 2dMT(|x - x'|^{\delta} + |t-t'|^{\delta/2})
%		\end{aligned}
%	\end{equation}
%	since $\|v_d\|_{L^{\infty}(Q_T)}\leq d\|\wh{z}\|_{L^\infty(Q_T)} + \|\wh{u}\|_{L^\infty(Q_T)} \leq 2dMT$. From \eqref{est_vd1} and \eqref{est_vd2} we get \eqref{vd_Holder} with $C_5 = \max\{dC_1, C_2, dM\}$ and $\delta = \min\{\alpha, \beta\}$.
\end{proof}

\begin{lemma}\label{lem:zhat}
Define
		\begin{equation*}
			\wh{z}(x,t):= \int_0^tz(x,s)ds.
		\end{equation*}
		Then, $\wh{z}$ solves
		\begin{equation*}
			\pa_t\wh{z} - d\Delta\wh{z} = \sum_{i=1}^{N}u_{i0} + K_0t \; \text{ in } Q_T, \qquad \nx\wh{z}\cdot\nu = 0\; \text{ on } \pa\Omega\times (0,T), \qquad \wh{z}(x,0) = 0 \; \text{ in } \Omega
		\end{equation*}
		and consequently $\wh{z}$ is H\"older continuous in $Q_T$ with an exponent $\beta \in (0,1)$. % \textcolor{blue}{Value? Reference, Using $L^{\infty}$, Generalise to large $L^p$?}
	\end{lemma}
	\begin{proof}
		The boundary condition and initial data of $\wh{z}$ are obvious. Integrating the equation \eqref{sysz} of $z$ over $(0,t)$, we have
		\begin{equation*}
			z(x,t) - d\Delta \int_0^tz(x,s)ds = z(x,0) + K_0t
		\end{equation*}
		and hence
		\begin{equation*}
			\pa_t\wh{z} - d\Delta \wh{z} = z(x,0) = \sum_{i=1}^{N}u_{i0} + K_0t.
		\end{equation*}
The comparison principle implies
\begin{equation}\label{zhatLinfty}
			\sup_{Q_T}|\wh z| \leq (M+K_0T)T, \qquad \text{recalling}\quad M:= \sum_{i=1}^{N}\|u_{i0}\|_{L^\infty(\O)}.
\end{equation}
%		\red{This uniform estimate is WRONG! $\|\wh{z}\|_{L^\infty(Q_T)}$ should depends on $T$}
Therefore, H\"older continuity of $\wh{z}$ follows from classical parabolic theory (see e.g. \cite{Nash58} for $\Omega = \R^n$ or \cite[Theorem 4.8]{Lie96} for bounded domains). More precisely, there exists $\beta\in (0,1)$ and $C_1$ depending only on $\Omega$, $n$, $N$ and $M$ such that
		\begin{equation*}
			|\wh{z}(x,t) - \wh{z}(x',t')| \leq C_1(1+T^2)(|x-x'|^\beta + |t-t'|^{\beta/2}) \quad \text{ for all } \quad (x,t), (x',t')\in Q_T.
		\end{equation*}
		The term $1+T^2$ on the right hand side is because of the $L^\infty$ bound of $\wh{z}$ in \eqref{zhatLinfty}. %\textcolor{blue}{Reference clarifying $T$ dependency} 
	\end{proof}

\begin{lemma}\label{lem:uhat_eq}
	Define as in \eqref{vd}
	\begin{equation}\label{uhat_def}
		\wh{u}(x,t):= \int_0^t\sum_{i=1}^{N}d_iu_i(x,s)ds.
	\end{equation}
	Then, $\wh{u}$ solves
	\begin{equation}\label{uhat}
		b \pa_t\wh{u} - \Delta \wh{u} = \sum_{i=1}^{N}u_{i0} + K_0t, \qquad \nx \wh{u} \cdot \nu = 0, \qquad \wh{u}(x,0) = 0,
	\end{equation}
	where
		\begin{equation*}
			b(x,t) = \frac{\sum_{i=1}^{N}u_i(x,t)}{\sum_{i=1}^Nd_iu_i(x,t)}
		\end{equation*}
	satisfies the bound
	\begin{equation}\label{bbound}
		\frac{1}{\max\{d_i\}} \leq b(x,t) \leq \frac{1}{\min\{d_i\}} ,\qquad \text{ for all } \quad (x,t)\in Q_T.
	\end{equation}
\end{lemma}
\begin{proof}
	Summing all equations in \eqref{eq1} and using assumption \eqref{A3eq} leads to
	\begin{equation*}
		\pa_t\sum_{i=1}^N u_i - \Delta \sum_{i=1}^Nd_iu_i = K_0.
	\end{equation*}
	Integrating this equation over $(0,t)$ yields
	\begin{equation*}
		\sum_{i=1}^{N}u_i(x,t) - \Delta \wh{u} = \sum_{i=1}^{N}u_{i0}(x) + K_0t
	\end{equation*}
	which implies \eqref{uhat} since $\pa_t\wh{u} = \sum_{i=1}^{N}d_iu_i$. The bounds of $b(x,t)$ in \eqref{bbound} follows easily from the non-negativity of $u_i$.
%	which leads to \eqref{uhat}
%	which satisfies, thanks to the non-negativity of $u_i$,
%	\begin{equation*}
%		0 < \frac{1}{\max d_i} \leq b(x,t) \leq \frac{1}{\min d_i} < +\infty, \quad \text{ for all } \quad (x,t) \in Q_T.
%	\end{equation*}
%	\blue{We now show that solutions to the parabolic equation \eqref{uhat} is H\"older continuous using the standard technique.} \red{This results require a bound $\|\wh{u}\|_{L^{\infty}(Q_T)}$.}
\end{proof}
\begin{lemma}\label{lem:uhat}
	The function $\wh{u}$ defined in \eqref{uhat_def} is bounded in $Q_T$, more precisely $\|\wh{u}\|_{L^{\infty}(Q_T)} \leq d(M+K_0T)T$. Moreover, $\wh{u}$ is H\"older continuous with some exponent $\alpha \in (0,1)$, i.e.
	\begin{equation}\label{uhat_Holder}
		|\wh{u}(x, t) - \wh{u}(x', t')| \leq C_2(1+T^2)(|x - x'|^{\alpha} + |t - t'|^{\alpha/2}) \qquad \text{ for all }\quad (x,t), (x',t') \in Q_T,
	\end{equation}
	where $C_2$ and $\alpha$ depend on $\Omega$, $M$, $n$, $K_0$ and $d_i$. %\textcolor{blue}{which $\alpha$'s?}
\end{lemma}
\begin{remark}
	The local H\"older continuity of $\wh{u}$ follows from a well-known result in \cite{KS80} since $\wh{u}$ also solves the following parabolic equation of non-divergence form
	\begin{equation*}
		\pa_t \wh{u} - \frac{1}{b}\Delta \wh{u} = \frac{1}{b}\sum_{i=1}^{N}u_{i0} + \frac{1}{b}K_0t.
	\end{equation*}
	The results in \cite{KS80} are based on a probabilistic approach. Here, we provide an alternative proof using \cite{LSU88} and \cite{Nit11}. A similar result was presented in \cite[Proposition 3.1]{CGV}, in which the authors proved H\"older continuity for a homogeneous equation.
\end{remark}
%\begin{remark}\label{diss_mass_2}
%	If condition \eqref{A3} is replaced by \eqref{A3bis}, then $\wh{u}$ only satisfies the inequality $b\pa_t \wh{u} - \Delta \wh{u} \leq \sum_{i=1}^{N}u_{i0}$, or in other words $\wh{u}$ is a only a {\normalfont nonnegative subsolution} to a parabolic equation of non-divergence form, and therefore the H\"older continuity of $\wh{u}$ is unclear.
%\end{remark}
\begin{proof}[Proof of Lemma~\ref{lem:uhat}]
	Since $v_i \geq 0$ and $d\geq d_i$ for all $i=1,\ldots, N$, we use $v_d =  d\wh{z} - \wh{u}$ to get
	\begin{equation*}
		0 \leq \wh{u} = d\wh{z} - v_d \leq d\wh{z}.
	\end{equation*}
	Therefore, thanks to Lemma \ref{lem:zhat},
	\begin{equation}\label{Linfbound}
		\sup_{Q_T}|\wh{u}|  \leq d\sup_{Q_T}|\wh{z}| \leq d(M+K_0T)T.
	\end{equation}
	
	The proof of H\"older continuity is technical and lengthy, and therefore, we postpone it to the Appendix \ref{appendix}.
	\end{proof}

\begin{lemma}\label{lem:vdreg}
The function $v_d= d\wh{z} - \wh{u}$ as defined in \eqref{zvd} is H\"older continuous in $Q_T$, i.e. there exist $C_5>0$ and $\delta \in (0,1)$ such that
	\begin{equation}\label{vd_Holder}
		|v_d(x,t) - v_d(x', t')| \leq C_5(1+T^2)(|x -x'|^\delta + |t-t'|^{\delta/2}) \quad \text{ for all } \quad (x,t), (x', t') \in Q_T.
	\end{equation}
\end{lemma}
\begin{proof}
%	Integrating \eqref{vi} on $(0,t)$ and using $v_i(x,0) = 0$ give
%	\begin{equation*}
%		v_i(x,t) - d\Delta \int_0^tv_i(x,s)ds = \int_0^tu_i(x,s)ds.
%	\end{equation*}
%	Hence
%	\begin{equation*}
%		\begin{aligned}
%			v_d(x,t) &=\sum_{i=1}^{N}(d-d_i)v_i(x,t)\\
%			&= d\int_0^t\sum_{i=1}^N(d-d_i)\Delta v_i(x,s)ds + \int_0^t\sum_{i=1}^{N}(d-d_i)u_i(x,s)ds\\
%			&= d\int_0^t\left[d\Delta \sum_{i=1}^Nv_i(x,s) + z(x,s) - \sum_{i=1}^{N}\pa_sv_i(x,s)\right]ds + \int_0^t\sum_{i=1}^N(d-d_i)u_i(x,s)ds\\
%			&= d\int_0^t\left[z(x,s) - \sum_{i=1}^Nu_i(x,s)\right]ds + \int_0^t\sum_{i=1}^N(d-d_i)u_i(x,s)ds\\
%			&= d\int_0^tz(x,s)ds - \int_0^t\sum_{i=1}^{N}d_iu_i(x,s)ds\\
%			&= d\wh{z} - \wh{u}.
%		\end{aligned}
%	\end{equation*}
The H\"older continuity follows clearly from the Lemmas \ref{lem:zhat} and \ref{lem:uhat}. In the following, we provided 
explicit estimates for the constants. First, for $|x -x'|, |t - t'| \leq 1$, we estimate
	\begin{equation}\label{est_vd1}
		\begin{aligned}
			|v_d(x,t) - v_d(x',t')| &\leq d|\wh{z}(x,t) - \wh{z}(x',t')| + |\wh{u}(x,t) - \wh{u}(x',t')|\\
			&\leq \max\{dC_1,C_2 \}(1+T^2)(|x-x'|^{\delta} + |t-t'|^{\delta/2}),
		\end{aligned}
	\end{equation}
	where $\delta = \min\{\alpha, \beta\}$ with $\beta$ and $\alpha$ as in Lemmas \ref{lem:zhat} and \ref{lem:uhat}. %\textcolor{blue}{Why necessary if already $\wh{z}$ and $\wh{u}$ H\"older?} \red{Here is just a precise proof of the fact that the sum of two H\"older continuous functions is continuous.} 
	When $|x - x'| \geq 1$ or $|t - t'| \geq 1$, we estimate
	\begin{equation}\label{est_vd2}
		\begin{aligned}
		|v_d(x,t) - v_d(x',t')| &\leq 2\|v_d\|_{L^{\infty}(Q_T)}(|x - x'|^{\delta} + |t-t'|^{\delta/2})\\
		&\leq 2d(M+K_0T)T(|x - x'|^{\delta} + |t-t'|^{\delta/2})
		\end{aligned}
	\end{equation}
	since $\|v_d\|_{L^{\infty}(Q_T)}\leq d\|\wh{z}\|_{L^\infty(Q_T)} + \|\wh{u}\|_{L^\infty(Q_T)} \leq 2d(M+K_0T)T$. From \eqref{est_vd1} and \eqref{est_vd2}, we get \eqref{vd_Holder} with $\delta = \min\{\alpha, \beta\}$.
\end{proof}

\begin{lemma}\label{lem:gradvd}
	Let $|U| = \sup_{Q_T}\max_{i=1,\ldots,N}|u_i|$. Then, there exists a constant $C_6 >0$ such that
	\begin{equation}\label{gradvd}
		\sup_{Q_T}|\nx v_d| \leq C_6\,(1+T^2)^{\frac{1}{2-\delta}}|U|^{\frac{1-\delta}{2-\delta}}.
	\end{equation}
\end{lemma}
\begin{proof}
	From Lemma \ref{lem:vdreg}, we know that $v_d$ is H\"older continuous. In particular
	\begin{equation*}
		|v_d(x,t) - v_d(x',t)| \leq C_5\,(1+T^2)|x - x'|^{\delta}, \qquad \text{ for all } \quad (x,t), (x',t)\in Q_T.
	\end{equation*}
	By the initial data $v_i(x,0) = 0$ and $\sup_{Q_T}|u_i| \leq |U|$ (i.e. by the definition of $|U|$), we can apply Lemma \ref{lem:main} to \eqref{vi} and obtain
	\begin{equation*}
		\sup_{Q_T}|\nx v_d| \leq B(C_5\,(1+T^2))^{\frac{1}{2-\delta}}|U|^{\frac{1-\delta}{2-\delta}},
	\end{equation*}
	whence \eqref{gradvd} with $C_6 = BC_5^{\frac{1}{2-\delta}}$.
\end{proof}

\begin{lemma}\label{lem:gradud}
	There exists a constant $C_7 >0$ such that
	\begin{equation*}
		\sup_{Q_T}|\nx u_d| \leq C_7\left(1+|U|^{\frac{3+\varepsilon}{2}} \right).
	\end{equation*}
\end{lemma}
\begin{proof}
	Firstly by the definition of $u_d$ in \eqref{vdeq} it follows
	\begin{equation*}
		\sup_{Q_T}|\nx u_d| \leq \sum_{i=1}^{N}(d-d_i)\sup_{Q_T}|\nx u_i|.
	\end{equation*}
	In order to apply Lemma \ref{lem:main} to equation of $u_i$ we first observe that
	\begin{equation*}
		|u_i(x,t) - u_i(x',t)| \leq 2\sup_{Q_T}|u_i| \leq 2|U||x-x'|^0, \qquad \text{ for all } (x,t), (x',t)\in Q_T.
	\end{equation*}
	The right hand sides in \eqref{eq1} can be estimated by using the growth assumption \eqref{A5}
	\begin{equation*}
		\sup_{Q_T}|f_i(u)| \leq K(1+ \sup_{Q_T}|u|^{2+\varepsilon}) \leq KN^{2+\varepsilon}(1+|U|^{2+\varepsilon}).
	\end{equation*}
	Now we can apply Lemma \ref{lem:main} to the equation of $u_i$ with the exponent $\gamma = 0$ to obtain
	\begin{equation*}
		\begin{aligned}
		\sup_{Q_T}|\nx u_i| &\leq C\|u_{i0}\|_{C^1(\Omega)} + B(2|U|)^{1/2}\left(KN^{2+\varepsilon}(1+|U|^{2+\varepsilon})\right)^{1/2}\\
		&\leq C\|u_{i0}\|_{C^1(\Omega)} + C_8\left(1+|U|^{\frac 32 + \frac \varepsilon2}\right)
		\end{aligned}
	\end{equation*}
	where $C_8$ depends on $B, K$ and $N$.
	Hence
	\begin{equation*}
		\sup_{Q_T}|\nx u_d| \leq \max\{d_i\}\sum_{i=1}^{N}\sup_{Q_T}|\nx u_i| \leq C_7\left(1+ |U|^{\frac{3+\varepsilon}{2}}\right),
	\end{equation*}
	which is the claim of Lemma \ref{lem:gradud} where $C_7$ depends on $\|u_{i0}\|_{C^1(\Omega)}$, $d_i$, $N$ and $C_8$.
\end{proof}
	{\color{black}
\begin{lemma}\label{lem:Lapvd}
	There exists a constant $C_9>0$ such that
	\begin{equation*}
		\sup_{Q_T}|\Delta v_d| \leq C_9(1+T)\left(1+|U|^{\frac{3+\varepsilon}{4} + \frac{1-\delta}{2(2-\delta)}}\right).
	\end{equation*}
\end{lemma}
\begin{proof}
	By definition $v_d = \sum_{i=1}^{N}(d-d_i)v_i$, we have
	\begin{equation*}
		\begin{aligned}
		\pa_t v_d - d\Delta v_d &= \pa_t\sum_{i=1}^{N}(d-d_i)v_i  - d\Delta \sum_{i=1}^{N}(d-d_i)v_i
		= \sum_{i=1}^{N}(d-d_i)(\pa_t v_i - d\Delta v_i)\\
		&= \sum_{i=1}^{N}(d-d_i)u_i = u_d.
		\end{aligned}
	\end{equation*}	
	
	Let $k>0$ to be chosen later, we define $\eta(x,t) = e^{kt}v_d(x,t)$. It follows from the equation of $v_d$ that
	\begin{equation}\label{eq_eta}
		\pa_t\eta - d\Delta \eta = e^{kt}(kv_d + u_d), \qquad \nabla_x \eta \cdot \nu = 0,\qquad \eta(x,0)= v_d(x,0) = 0.
	\end{equation}
	By Duhamel's formula
	\begin{equation*}
		\eta(x,t) = \int_0^te^{(t-s)d\Delta}e^{ks}(kv_d(x,s) + u_d(x,s))ds.
	\end{equation*}
	Changing back to $v_d$ we get
	\begin{equation*}
		v_d(x,t) = \int_0^te^{-k(t-s)}e^{(t-s)d\Delta}(kv_d(x,s) + u_d(x,s))ds.
	\end{equation*}
	We now apply the second order estimate for semigroup, see e.g. \cite[Eq. (2.39)]{Mor83}
	\begin{equation}\label{Delta_vd}
		\|e^{td\Delta }f\|_{C^2(\Omega)} \leq Ct^{-1/2}\|f\|_{C^1(\Omega)}, \qquad \text{for}\quad t\in(0,T]
	\end{equation}
	to have
	\begin{align*}
		\sup_{Q_T}|\Delta v_d| \leq C\sup_{t\in(0,T)}\|v_d(t)\|_{C^2(\Omega)} &\leq C\int_0^te^{-k(t-s)}(t-s)^{-1/2}[k\|v_d(s)\|_{C^1(\Omega)} + \|u_d(s)\|_{C^1(\Omega)}]ds,
	\end{align*}
	where in this proof we always denote by $C$ a generic constant {\it independent of $T$}. We estimate
	\begin{align*}
	\int_0^te^{-k(t-s)}(t-s)^{-1/2}k\|v_d(s)\|_{C^1(\Omega)}ds
	&\leq \sup_{s\in (0,T)}\|v_d(s)\|_{C^1(\Omega)}k\int_0^te^{-k(t-s)}(t-s)^{-1/2}ds\\
	&\leq C\sup_{Q_T}(|v_d| + |\nabla v_d|)\sqrt{k} \int_0^\infty s^{-1/2}e^{-s}ds\\
	&\leq C\sup_{Q_T}(|v_d| + |\nabla v_d|)\sqrt{k}
	\end{align*}
	and similarly
	\begin{equation*}
		\int_0^te^{-k(t-s)}(t-s)^{-1/2}\|u_d(s)\|_{C^1(\Omega)}ds \leq C\sup_{Q_T}(|u_d| + |\nabla u_d|)\frac{1}{\sqrt{k}}.
	\end{equation*}
	By choosing $\sqrt{k} = \sup_{Q_T}(|u_d| + |\nabla u_d|)^{1/2}\sup_{Q_T}(|v_d| + |\nabla v_d|)^{-1/2}$, we obtain from \eqref{Delta_vd}
	\begin{equation*}
		\sup_{Q_T}|\Delta v_d| \leq C\sup_{Q_T}(|u_d| + |\nabla u_d|)^{1/2}\sup_{Q_T}(|v_d| + |\nabla v_d|)^{1/2}.
	\end{equation*}
	Therefore, by using Lemma~\ref{lem:vdreg}, i.e. $\sup_{Q_T}|v_d| \leq 2d(M+K_0T)T$ and Lemmas \ref{lem:gradvd}, \ref{lem:gradud}, we finally get
	\begin{align*}
		\sup_{Q_T}|\Delta v_d| &\leq C\left(|U| + 1 + |U|^{\frac{3+\varepsilon}{2}}\right)^{1/2}\left(1 + T^2 +  (1+T^2)^{\frac{1}{2-\delta}}|U|^{\frac{1-\delta}{2-\delta}}\right)^{1/2}\\
		&\leq C(1+ T)\left(1+ |U|^{\frac{3+\varepsilon}{4}+\frac{1-\delta}{2(2-\delta)}}\right).
	\end{align*}
\end{proof}
\begin{remark}
	In the case $\Omega = \mathbb R^d$, we can apply Lemma \ref{lem:main} to the equation of $\pa_{x_j}v_d$ to obtain Lemma \ref{lem:Lapvd} immediately. This argument seems to not directly work in the case of a bounded domain since the Neumann boundary condition is not satisfied by $\pa_{x_j}v_d$.
\end{remark}
}

%\begin{lemma}
%	There exists a constant $C$ depending on the diffusion coefficients, the initial data such that
%	\begin{equation*}
%		\max_{Q_T}|U| \leq C_6.
%	\end{equation*}
%	Consequently, the local classical solution to \eqref{eq1} is global.
%\end{lemma}
We are now ready to prove Theorem \ref{thm:main} with \eqref{A3eq}.
\begin{proof}[Proof of Theorem \ref{thm:main} with \eqref{A3eq}]
	From \eqref{zvd} we have
	\begin{equation*}
		|U| = \max_{i=1,\ldots, N}\sup_{Q_T}|u_i| \leq \sum_{i=1}^{N}\sup_{Q_T}|u_i| \leq \sup_{Q_T}(|z| + |\Delta v_d|).
	\end{equation*}
	By using Lemmas \ref{lem:boundz} and \ref{lem:Lapvd}, it follows
	\begin{equation}\label{est.U}
	\begin{aligned}
		|U| &\leq (M + K_0T)+ C_9(1+T)\left(1+|U|^{\frac{3+\varepsilon}{4} + \frac{1-\delta}{2(2-\delta)}}\right)\\ &\leq C_{11}\left[1+ T+(1+T)|U|^{\frac{3+\varepsilon}{4} + \frac{1-\delta}{2(2-\delta)}}\right],
	\end{aligned}
	\end{equation}
	where $C_{11}$ depends only on $N$, $C_9$ and $M$. We choose $\varepsilon$ small enough such that
	\begin{equation*}
		\lambda:= \frac{3+\varepsilon}{4} + \frac{1-\delta}{2(2-\delta)} < 1 \quad \text{ or equivalently } \quad \varepsilon < \frac{\delta}{2 - \delta}.
	\end{equation*}
	By using Young's inequality we have
	\begin{equation*}
		C_{11}(1+T)|U|^{\lambda} \leq \frac{|U|}{2} + (1-\lambda)(2\lambda C_{11})^{\frac{\lambda}{1-\lambda}}(1+T)^{\frac{1}{1-\lambda}}
	\end{equation*}
	and therefore it follows from \eqref{est.U} 
%	we obtain from \eqref{est.U} with the help of Young's inequality
	\begin{equation*}
		|U| \leq 2C_{11}\left[1+T\right] + 2(1-\lambda)(2\lambda C_{11})^{\frac{\lambda}{1-\lambda}}(1+ T)^{\frac{1}{1-\lambda}} \leq C_{12}(1+T^{\frac{1}{1-\lambda}}),
	\end{equation*}
	where $C_{12}$ depends only on $C_{11}$, $\varepsilon$ and $\delta$.
%	which leads to the estimate
%	\begin{equation*}
%		|U| \leq C_0:= 2C_{12}.
%	\end{equation*}
%	Since $C_0$ does not depend on $T$, we obtain the global existence of classical solution and also the uniform in time $L^\infty$-bound. 
	The uniqueness follows immediately thanks to the $L^\infty$-bound and the local Lipschitz continuity of the nonlinearities.
\end{proof}

\begin{remark}\label{remark:K0t}
	It's straightforward that all the arguments of this section are still valid in case $K_0$ is replace by a function $K_0(t)$ which is continuous on $[0,\infty)$, except that the $L^\infty$-norm of the solution might grow faster than polynomial.
\end{remark}

\section{Proof of Theorem \ref{thm:main} with condition \eqref{A3}}\label{sec:control}
\begin{proof}[Proof of Theorem \ref{thm:main} with \eqref{A3}]
Our main idea is that with a suitable change of unknowns, and especially {\it adding one more appropriate equation}, we can transform a system with the mass control condition \eqref{A3} into a system with condition \eqref{A3eq}, that keeps the essential features \eqref{A4} and \eqref{A5}.

%\medskip
%From now on, for $1\leq p \leq \infty$ we write $u \in L^p_{x,t}$ if $u\in L^p(0,T;L^p(\Omega))$ for all $0 < T < T_{\max}$ and {\it more importantly}
%\begin{equation*}
%	\sup_{0<T<T_{\max}}\|u\|_{L^p(0,T;L^p(\Omega))} < +\infty.
%\end{equation*}

\medskip
We define 
\[w_i(x,t) = e^{-K_1t}u_i(x,t) \quad \text{ or equivalently } \quad
u_i(x,t) = e^{K_1 t}w_i(x,t)
\]
and $w = (w_1, \ldots, w_N)$. Direct computations give
\begin{align*}
	\pa_t w_i &= e^{-K_1t}(\pa_t u_i - K_1u_i)\\
	&= e^{-K_1t}(d_i\Delta u_i + f_i(u) - K_1u_i)\\
	&= d_i\Delta w_i + g_i(w)
\end{align*}
where 
\begin{equation}\label{gi}
	g_i(w) = e^{-K_1t}(f_i(u) - K_1e^{K_1t}w_i).
\end{equation}
Note that
\begin{align}\label{m+1}
	\sum_{i=1}^{N}g_i(w) &= e^{-K_1t}\sum_{i=1}^{N}(f_i(u) - K_1 u_i)\nonumber
%	&\leq e^{-Kt}\left[C_0 + \sum_{i=1}^{N}(C_1 - K\alpha_i)u_i\right]\nonumber\\
	\leq e^{-K_1t}K_0
\end{align}
due to the assumption \eqref{A3}.

Introduce a new unknown $w_{N+1}: \Omega\times(0,T_{\max}) \to \mathbb \R_+$ which solves
\begin{equation}\label{wm1}
	\pa_t w_{N+1} - \Delta w_{N+1} = K_0e^{-K_1 t} - \sum_{i=1}^{N}g_i(w) =: g_{N+1}(w) \geq 0
\end{equation}
with homogeneous Neumann boundary condition $\nabla w_{N+1}\cdot \nu = 0$ and zero initial data $w_{N+1}(x,0) = w_{N+1,0}(x) = 0$ for $x\in\Omega$. With a slight abuse of notation we write the new vector of concentrations $\wt w = (w_1, \ldots, w_N, w_{N+1})$ and the nonlinearities $g_i(\wt w):= g_i(w_1, \ldots, w_N)$ for all $i=1,\ldots, N$ while $g_{N+1}(\wt w) = K_0e^{-K_1 t} - \sum_{i=1}^Ng_i(w)$. We have arrived at the following system
\begin{equation}\label{Sw}%\tag{Sw}
	\begin{cases}
		\pa_t w_i - d_i\Delta w_i = g_i(\wt w), &(x,t)\in \Omega \times (0,T_{\max}), \quad i=1,\ldots, N+1,\\
		\nabla w_i \cdot \nu = 0, &(x,t)\in \pa\Omega\times (0,T_{\max}), \quad i=1,\ldots, N+1,\\
		w_i(x,0)  =  u_{i,0}(x), & x\in\Omega, \quad i=1,\ldots, N,\\
		w_{N+1}(x,0) = w_{N+1,0}(x) = 0, & x\in\Omega.
	\end{cases}
\end{equation}
It's obvious to check that the nonlinearities $g_i, i = 1,\ldots, N+1$ satisfy the assumption \eqref{A4}. Moreover, due to the definition $w_i(x,t) = e^{-K_1t}u_i(x,t)$, it follows from \eqref{A5} and \eqref{gi} the growth control
\begin{equation}\label{growth_new}
\begin{aligned}
	|g_i(\wt w)| &\leq e^{-K_1t}(|f_i(u)| + K_1e^{K_1t}|w_i|)\\
	&\leq e^{-K_1t}(K(1+ |u|^{2+\varepsilon}) + K_1e^{K_1t}|w_i|)\\
	&\leq Ce^{(1+\varepsilon)K_1 T_{\max}}(1+|\wt{w}|^{2+\varepsilon}).
\end{aligned}
\end{equation}
%\[
%e^{-K_1t}|u|^{2+\epsilon} = e^{-K_1t}|e^{K_1t}w|^{2+\epsilon} =e^{-K_1t} e^{K_1(2+\epsilon)t}|w|^{2+\epsilon}=e^{K_1(1+\epsilon)t}|w|^{2+\epsilon}
%\]
Moreover, the nonlinearities of \eqref{Sw} satisfies the  condition \eqref{A3eq} (with $K_0$ is replaced by a continuous function in $t$), i.e.
\begin{equation}\label{Sw_conservation}
	\sum_{i=1}^{N+1}g_i(\wt w) = K_0e^{-K_1 t}
\end{equation}
thanks to \eqref{wm1}.

Now we can apply the results in Section \ref{sec:thm} (Remark \ref{remark:K0t}) to get that \eqref{Sw} has a global classical solution $\wt{w}$. Changing back to the original unknowns $u_i(x,t) = e^{K_1t}w_i(x,t)$ for $i=1,\ldots, N$, we obtain finally the global existence of classical solution to \eqref{eq1}. 

\medskip
Moreover, in case $K_0= 0$ and $K_1 < 0$, the growth condition \eqref{growth_new} can be estimated further as
\begin{equation*}
|g_i(\wt w)| \leq C(1+|\wt w|^{2+\varepsilon})
\end{equation*}
and \eqref{Sw_conservation} becomes
\begin{equation*}
\sum_{i=1}^{N+1}g_i(\wt w) = 0.
\end{equation*}
Therefore, from Section \ref{sec:thm}, we know that the solution $\wt w$ to \eqref{Sw} grows at most polynomially in time, i.e. for all $i=1,\ldots, N+1$
\begin{equation*}
	\|w_i(t)\|_{L^{\infty}(\Omega)} \leq C(1+t^{\zeta})
\end{equation*}
for some $C, \zeta>0$. Therefore, if $K_0 = 0$ and $K_1 < 0$ we get
\begin{equation*}
	\|u_i(t)\|_{L^\infty(\Omega)} = e^{K_1t}\|w_i(t)\|_{L^\infty(\Omega)} \leq C(1+ e^{K_1t}t^\zeta) \leq C(1+e^{-\mu t})
\end{equation*}
for some $K_1 < -\mu < 0 $, which completes the proof of Theorem \ref{thm:main}.
\end{proof}

	\section{Applications}\label{sec:appl}	
	\subsection{Quadratic reversible reactions}
	We consider the reversible chemical reaction
	\begin{equation*}
		A_1 + A_2 \leftrightarrows A_3 + A_4
	\end{equation*}
and denote by $u_i(x,t)$ the concentrations of the substances $A_i$ at $x\in\O$ and $t>0$. 
Using the mass-action law, we obtain the following reaction-diffusion system
	\begin{equation}\label{sys:quad}
		\begin{cases}
			\pa_t u_1 - d_1\Delta u_1 = -u_1u_2 + u_3u_4 =: f_1(u), &\quad(x,t)\in Q_T,\\
			\pa_t u_2 - d_2\Delta u_2 = -u_1u_2 + u_3u_4=: f_2(u), &\quad(x,t)\in Q_T,\\
			\pa_t u_3 - d_3\Delta u_3 = +u_1u_2- u_3u_4=: f_3(u), &\quad(x,t)\in Q_T,\\
			\pa_t u_4 - d_4\Delta u_4 = +u_1u_2- u_3u_4=: f_4(u), &\quad(x,t)\in Q_T,\\
			\nx u_i \cdot \nu = 0, &\quad(x,t)\in \partial\O\times (0,T),\\
			u_i(x,0) = u_{i0}(x), &\quad x\in\O,
		\end{cases}
	\end{equation}
where we have normalised the forward and backward reaction rate constants for the sake of simplicity. We assume positive diffusion coefficients $d_i >0$. System \eqref{sys:quad} was studied extensively in the literature: Weak solutions were shown in all dimensions \cite{DFPV07} while the global strong (classical) solutions were shown in \cite{GV10} or \cite{CDF14} in one or two dimensions. In higher dimensions, \cite{CDF14} showed global classical solutions under the assumption that the diffusion coefficients are sufficiently close to each other (depending on the space dimension). The question of global classical without restriction on diffusion coefficients was solved in \cite{Kan90} for the case $\Omega = \R^n$ and recently reproved in \cite{CGV} and \cite{Sou18}. We remark that besides satisfying Assumptions \eqref{A3}, \eqref{A4} and \eqref{A5}, the nonlinearities of \eqref{sys:quad} have an additional feature, that is an \textit{entropy inequality}, i.e.
	\begin{equation}\label{entropy}
		\sum_{i=1}^{4}f_i(u)\log u_i \leq 0 \quad \text{ for all } \quad u \in \R_+^4,
	\end{equation}
which  plays an important role in the analysis of \cite{CGV,Sou18}. This entropy property is also a key in studying the convergence to chemical equilibrium for \eqref{sys:quad}. Due to the homogeneous Neumann boundary condition, system \eqref{sys:quad} possesses precisely three linear independent conservation laws for all $t>0$
	\begin{equation*}
		\int_{\O}[u_1(x,t) + u_3(x,t)]dx = \int_{\O}[u_{10}(x) + u_{30}(x)]dx =: M_{13},
	\end{equation*}
	\begin{equation*}
		\int_{\O}[u_2(x,t)+u_3(x,t)]dx = \int_{\O}[u_{20}(x) + u_{30}(x)]dx =: M_{23},
	\end{equation*} 
	\begin{equation*}
	\int_{\O}[u_2(x,t)+u_4(x,t)]dx = \int_{\O}[u_{20}(x) + u_{40}(x)]dx =: M_{24}.
	\end{equation*}
	Denote by $\mb{M} = (M_{13}, M_{23}, M_{24})$ the initial mass vector. By straightforward computations, for any fixed positive initial mass vector $\mb{M}\in \R_+^3$, there exists a unique positive constant equilibrium $u_\infty = (u_{1\infty}, u_{2\infty}, u_{3\infty}, u_{4\infty}) \in \R_+^4$ satisfying
	\begin{equation*}
		u_{1\infty}u_{2\infty} = u_{3\infty}u_{4\infty}, \quad u_{1\infty} + u_{3\infty} = M_{13}, \quad u_{2\infty} + u_{3\infty} = M_{23}, \quad u_{2\infty} + u_{4\infty} = M_{24}.
	\end{equation*}
	Note that the equilibrium is determined only by the initial mass vector rather than the precise initial data. It was proved in several works, e.g. \cite{DF08,FT17,FT17a,PSZ16}, that any solution (renormalised, weak or strong) to \eqref{sys:quad} with initial mass $\mb{M}$ converges exponentially in $L^1$-norm to the equilibrium $u_\infty$ defined above, i.e.
	\begin{equation}\label{L1-converge}
		\sum_{i=1}^{4}\|u_i(t)-u_{i\infty}\|_{L^1(\O)} \leq Ce^{-\lambda t}, \qquad \text{ for all } \quad t\geq 0,
	\end{equation}
	where $C > 0, \lambda > 0$ are constants which can be computed explicitly. By applying Theorem \ref{thm:main}, 
%		and interpolation inequality
	we have
	\begin{theorem}\label{thm:quad}
		Let $\Omega$ be a bounded domain with smooth boundary $\pa\O$. Fix a positive initial mass vector $\mb{M} \in \R_+^3$. 
		
		Then, for any non-negative initial data $u_0 = (u_{i0})_{i=1\ldots4} \in L^{\infty}(\O)^4$ having the initial mass vector $\mb{M}$, there exists a unique global non-negative classical solution $u = (u_1, u_2, u_3, u_4)$ 
of \eqref{sys:quad} which converges exponentially to equilibrium in $L^\infty$-norm, i.e.
		\begin{equation}\label{conv.Lp}
			\sum_{i=1}^4\|u_i(t) - u_{i\infty}\|_{L^\infty(\O)} \leq C_0e^{-\lambda_0 t}, \qquad \text{ for all } \quad t \geq 0,
		\end{equation}
		where $C_0 > 0, \lambda_0 >0$ are positive constants which can be computed explicitly.
	\end{theorem}
	\begin{remark}
		The global existence of classical solution to \eqref{sys:quad} on the whole space $\mathbb R^n$ was proved in \cite{Kan90} and recently reproved in \cite{CGV,Sou18}. However, it seems that the results therein do not provide \textit{a priori} estimates in time for solutions, thus such convergence in $L^\infty$-norm \eqref{conv.Lp} does not follow immediately from their results.
	\end{remark}
	\begin{remark}
		Results similar to Theorem~\ref{thm:quad} are also valid for complex balanced networks of chemical reactions of substances $A_1, \ldots, A_S$ with quadratic nonlinearities, which means that each complex of the networks is of the form $A_i + A_j$ with $i,j = 1,\ldots, S$. The global existence and uniqueness of classical solution follows immediately from Theorem \ref{thm:main} with \eqref{genA4} replacing \eqref{A3}, see Remark \ref{thm:gen}. For the exponential convergence we need either the detailed or complex balanced condition, and the absence of boundary equilibria, see \cite{CC17} and \cite{Tan17} for more details.
	\end{remark}	
	\begin{proof}[Proof of Theorem~\ref{thm:quad}]
		From Theorem \ref{thm:main}, we have that the $L^\infty$-norm,
		\begin{equation*}
			\|u_i(T)\|_{L^\infty(\Omega)} \leq C_T, \qquad \text{ for all }\quad T> 0,
		\end{equation*}
		grows at most polynomially in $T$. Now using \eqref{L1-converge} and the interpolation inequality 		\begin{equation*}
			\|f\|_{L^p(\Omega)} \leq \|f\|_{L^\infty(\Omega)}^{(p-1)/p}\|f\|_{L^1(\Omega)}^{1/p}
		\end{equation*}
		for all $p \in (1,\infty)$,
we have
		\begin{equation}\label{Lp-converge}
			\begin{aligned}
			\|u_i(T) - u_{i\infty}\|_{L^p(\Omega)} &\leq \|u_i(T) - u_{i\infty}\|_{L^\infty(\Omega)}^{(p-1)/p}\|u_i(T) - u_{i\infty}\|_{L^1(\Omega)}^{1/p}
			\leq C_T^{(p-1)/p}C^{1/p}e^{-(\lambda/p)t}\\
			&\leq C_pe^{-\lambda_p T}
			\end{aligned}
		\end{equation}
		for some explicit constants $C_p>0$ and $0 < \lambda_p < \lambda/p$. It follows that for each $p\in (1,\infty)$ there exists $M_p > 0$ such that
		\begin{equation*}
			\sup_{t>0}\|u_i(t)\|_{L^p(\Omega)} \leq M_p, \qquad \text{ for all }\quad i =1,\ldots, 4.
		\end{equation*}
		Let $S_i(t) = e^{d_i\Delta t}$ be the heat semigroup with Neumann boundary condition. We use the $L^p-L^\infty$ estimate
		\begin{equation*}
			\|S_i(t)f\|_{L^{\infty}(\Omega)} \leq Ct^{-n/(2p)}\|f\|_{L^p(\Omega)},
		\end{equation*}
		the Duhamel formula
		\begin{equation*}
			u_i(t+1) = S(1)u(t) + \int_{0}^{1}S(1-s)f_i(u(t+s))ds,
		\end{equation*}
		and the quadratic growth $|f_i(u)| \leq C|u|^2$, to estimate
		\begin{equation*}
			\begin{aligned}
			\|u_i(t+1)\|_{L^\infty(\Omega)} &\leq C\|u_i(t)\|_{L^p(\Omega)} + C\int_0^1(1-s)^{-n/(2p)}\|u(t+s)\|_{L^{2p}(\Omega)}^2ds\\
			&\leq CM_p + CM_{2p}^2 \int_0^1(1-s)^{-n/(2p)}ds.
			\end{aligned}
		\end{equation*}
		By choosing $p > n/2$ which leads to $\int_0^1(1-s)^{-n/(2p)}ds \leq C$, we get a uniform in time bound of the $L^\infty$-norm,
		\begin{equation*}
			\|u_i(t)\|_{L^\infty(\Omega)} \leq M_\infty, \qquad \text{ for all } \quad t> 0 \quad \text{ and } \quad i=1,\ldots, 4.
		\end{equation*}
To prove exponential convergence in the $L^\infty$-norm, we observe that the equilibrium $u_\infty$ also solves the system \eqref{sys:quad}. Therefore, repeating the previous argument, we get
		\begin{equation*}
			\begin{aligned}
				\|u_i(t) - u_{i\infty}\|_{L^\infty(\Omega)} &\leq C\|u_i(t) - u_{i\infty}\|_{L^p(\Omega)} \!+ C\!\int_0^1\!(1-s)^{-n/(2p)}\|f_i(u(t+s)) - f_i(u_\infty)\|_{L^p(\Omega)}ds.
			\end{aligned}
		\end{equation*}
		It follows from the uniform in time bound of $L^\infty$-norm that
		\begin{equation*}
			\|f_i(u) - f_i(u_\infty)\|_{L^p(\Omega)} \leq C\|u - u_\infty\|_{L^p(\Omega)}.
		\end{equation*}
		Hence
		\begin{equation*}
					\begin{aligned}
						\|u_i(t) - u_{i\infty}\|_{L^\infty(\Omega)} &\leq C\|u_i(t) - u_{i\infty}\|_{L^p(\Omega)} + C\int_0^1(1-s)^{-n/(2p)}\|u(t+s) - u_\infty\|_{L^p(\Omega)}ds\\
						&\leq CC_pe^{-\lambda_pt} + CC_p\int_0^1(1-s)^{-n/(2p)}e^{-\lambda_p(t+s)}ds\\
						&\leq Ce^{-\lambda_p t}\left[1 + \int_0^1(1-s)^{-n/(2p)}\right]\\
						&\leq Ce^{-\lambda_p t},
					\end{aligned}
				\end{equation*}
		where we choose $p > n/2$ at the last steps. This finishes the proof of Theorem \ref{thm:quad}.
		\end{proof}	
	\subsection{Skew-symmetric Lotka-Volterra systems}
	A general Lotka-Volterra system with diffusion is of the form \eqref{LV}, i.e. 
	\begin{equation*}%\label{LV}
		\begin{cases}
		\pa_t {u}_i -  d_i\Delta {u}_i = \left(-\tau_i + \sum_{j=1}^{N}a_{ij}{u}_j\right){u}_i =: f_i(u), &(x,t)\in Q_T,\\
		\nx {u}_i \cdot \nu = 0, &(x,t)\in \pa\Omega\times (0,T),\\
		u_i(x,0) = u_{i0}(x), &x\in\O,
		\end{cases}
	\end{equation*}
	where $A = (a_{ij})\in \R^{N\times N}$ and $\tau = (\tau_1, \ldots, \tau_N)\in \R^N$. The system is called skew-symmetric when the matrix $A$ is skew-symmetric, i.e.
	\begin{equation*}
		A^{\top} + A = 0.
	\end{equation*}
	Lotka-Volterra systems of the form \eqref{LV} have been extensively investigated in the literature due to their wide application in biology, see e.g. \cite{Hen81,Rot84,LSY12,SY13}. In the case of skew-symmetric Lotka-Volterra systems, all existing results concern either weak solutions in all space dimensions or classical solutions in small (one and two) space dimensions. The results of our work imply classical solutions in all space dimensions. By the skew-symmetry of $A$, we obtain easily
	\begin{equation*}
		\sum_{i=1}^{N}f_i(u) = -\sum_{i=1}^N\tau_i u_i,
	\end{equation*}
	thus Assumption \eqref{A3} holds. It is straightforward that Assumptions \eqref{A4} and \eqref{A5} are also satisfied. We remark that \eqref{LV} does not possess an entropy inequality like \eqref{entropy}, and therefore the results of \cite{CGV} or \cite{Sou18} are not applicable in this case. By applying Theorem \ref{thm:main}, we have
	\begin{theorem}\label{thm:LV}
		Assume that $\O$ is a bounded domain with smooth boundary $\pa\O$. Assume moreover that the system \eqref{LV} is skew-symmetric, i.e. $A^\top + A = 0$. Then, for any non-negative initial data $u_{i0}\in L^\infty(\O)$, there exists a unique global classical solution to \eqref{LV}. Moreover, if $\tau_i > 0$ for all $i=1, \ldots, N$, we have
		\begin{equation*}
			\|u_i(t)\|_{L^\infty(\Omega)} \leq Ce^{-\mu t}
		\end{equation*}
		for some $\mu > 0$.
	\end{theorem}
	\begin{remark}
		Obviously, Theorem \ref{thm:LV} is also applicable in the case of ``sub-skew-symmetric" system, i.e. for $A + A^\top \leq 0$ in the sense that all elements of $A + A^\top$ are non-positive.
	\end{remark}
	\begin{proof}
		The global existence and uniqueness of classical solution follows from Theorem \ref{thm:main}. If $\tau_i > 0$ for all $i=1,\ldots, N$, we can take $K_1 = -\min\{\tau_1, \ldots, \tau_N\} < 0$ which implies
		\begin{equation*}
			\sum_{i=1}^Nf_i(u) = -\sum_{i=1}^N\tau_i u_i \leq K_1\sum_{i=1}^Nu_i
		\end{equation*}
		and therefore Theorem \ref{thm:main} directly applies yielding the exponential decay to zero. %The relative compactness of trajectories in $L^1(\O)$, regardless of the signs of $\tau_i$, was proved in \cite[Theorem 2]{PSY17}.
	\end{proof}
%	\begin{remark}
%		Unlike Theorem \ref{thm:quad}, we have in Theorem \ref{thm:LV} only the relative compactness of the trajectory but not the exponential convergence, therefore
%	\end{remark}
\appendix
\addtocontents{toc}{\protect\setcounter{tocdepth}{1}}
\section{Proof of H\"older continuity for $\wh u$}\label{appendix}
%	We remark that the proof we provide in this paper follows \cite[Proposition 3.1]{CGV} in which the authors proved the H\"older continuity for the equation
%	\begin{equation*}
%		\pa_t\Phi - d(x,t)\Delta \Phi = 0 \quad \text{ in } \quad \R^n
%	\end{equation*}
%	with $\Phi \in L^{\infty}((0,T)\times \R^n)$ and $\Delta \Phi \geq 0$, and $0 < \delta_\star \leq d(x,t) \leq \delta^\star < +\infty$ for a.e. $(x,t)\in (0,T)\times \R^n$. Due to the non-divergence form of this equation, the authors had to apply the De Giorgi method for the equation $\pa_t u - \delta_\star \Delta u + \mu = 0$ where $\mu$ is a nonnegative measure. By dividing the equation by $d(x,t)$ we obtain
%	\begin{equation*}
%		\frac{1}{d(x,t)}\pa_t\Phi - \Delta \Phi = 0
%	\end{equation*}
%	and thanks to $\pa_t\Phi = d(x,t)\Delta \Phi \geq 0$, we can apply the computations in this paper to obtain the H\"older continuity in the same way.}
%	
%	\vskip 2cm
	
%	\red{We will use the standard method in \cite{LSU88} for proving H\"older continuity of solutions to parabolic equations.}

%The proof of H\"older continuity of $\wh{u}$ is long, and therefore divided into several steps. In the following we first provide some useful preliminaries used frequently in the proof. Next, we derive some local energy estimates which are basically the essential estimates which are needed for a function to be H\"older continuous. The two main De Giorgi's lemmas are proved in the next subsection, which help to obtain the local H\"older continuity. 

In order to show H\"older continuity of $\wh u$ in Lemma \ref{lem:uhat}, we follow the approach from \cite[Chapter 2]{LSU88}. We remark that this approach gives only local H\"older continuity which is only sufficient in the case $\Omega = \mathbb R^n$. In the case of bounded domains $\Omega \subset \mathbb R^n$, we need to also prove boundary regularity. This will be done by reflection techniques, which was used e.g. in \cite{Nit11} for elliptic equations.

\subsection{Local H\"older continuity}\label{sec:loHo}\hfill\\
We fix $(x_0, t_0)\in Q_T$ and denote $$B_\varrho = \{x\in \Omega: |x-x_0| < \varrho\}$$ and $$Q(\varrho,\tau) = B_\varrho \times (t_0-\tau,t_0) =  \{(x,t): x\in B_\varrho, t_0 - \tau < t < t_0\}.$$

\begin{lemma}[Local H\"older continuity]\label{Holder-local}\hfill\\
For any $t_0>0$ and $\wh{\O} \subset\subset \O$, there exists $\rho_0>0$  with $d(\wh{\O}, \pa\O) \geq \rho_0$, and a constant $\alpha \in (0,1)$, such that 	for some constant $C>0$
	\begin{equation*}
		|\wh{u}(x,t) - \wh{u}(x', t')| \leq C\|\wh u\|_{L^\infty(Q_T)}(|x - x'|^{\alpha} + |t - t'|^{\alpha/2}),  \qquad \text{ for all } \quad (x,t), (x', t') \in \wh{\O}\times [t_0, T].
	\end{equation*}
\end{lemma}
\begin{proof}
According to \cite[Chapter 2, Section 7]{LSU88}, the H\"older continuity of $\wh u$ follows from the boundedness $\|\wh u\|_{L^\infty(Q_T)} \leq CT$ and the energy estimate of the following Lemma \ref{lem:ener}. 
\end{proof}
\begin{remark}
As an alternative proof of local H\"older continuity, note that the energy estimates \eqref{local-en} are sufficient to repeat the arguments in the proof of \cite[Proposition 3.1]{CGV}.
\end{remark}

\begin{lemma}[Local energy estimates]\label{lem:ener}\hfill\\
For any $k>0$, any $0 < \tau_2 < \tau_1 < 1$ and $0 < \varrho_2 < \varrho_1 < 1$ such that $Q(\varrho_1, \tau_1)\subset Q_T$, the following energy estimate holds
\begin{multline}\label{local-en}
%	\begin{aligned}
	\sup_{t_0-\tau_2 < t < t_0}\|(\wh u-k)_+\|_{L^2(B_{\varrho_2})}^2 + \int_{t_0-\tau_2}^{t_0}\|(\wh u-k)_+\|_{H^1(B_{\varrho_2})}^2ds\\
	\leq C\left[((\varrho_1-\rho_2)^{-2} + (\tau_1 - \tau_2)^{-1})\|(\wh u -k)_+\|_{L^2(Q(\varrho_1,\tau_1))}^2 + \int_{Q(\varrho_1,\tau_1)}\chi_{\{\wh u>k\}}dxds\right].
\end{multline}
\end{lemma}
\begin{proof} %[Proof of Lemma \ref{lem:ener}]
Let $\xi: Q_T \to [0,1]$ be a smooth cut-off function such that
		\begin{equation*}
			\xi(x,t) = \begin{cases}
				1 &\text{ if } \quad (x,t) \in Q(\vr_2, \tau_2)\\
				0 &\text{ if } \quad (x,t) \not\in Q(\vr_1, \tau_1).
			\end{cases}
		\end{equation*}
For simplicity we denote by $u^{(k)} = (\wh u-k)_+$. By multiplying equation \eqref{uhat} with $u^{(k)}\xi^2$ and integrating over $B_{\vr_1}\times (t_0-\tau_1, t)$, $t_0 - \tau_2 < t < t_0$ and by integration by parts, we calculate
		\begin{multline}\label{mul}
			%\begin{gathered}
				\int_{t_0-\tau_1}^t\int_{B_{\vr_1}}b\pa_t\wh u u^{(k)}\xi^2dx ds + \int_{t_0-\tau_1}^t\int_{B_{\vr_1}}|\nx u^{(k)}|^2\xi^2dxds\\
				= -2\int_{t_0-\tau_1}^t\int_{B_{\vr_1}} (\nx u^{(k)}\cdot \nx \xi) u^{(k)}\xi dxds + \int_{t_0-\tau_1}^t\int_{B_{\vr_1}}[u_0 + K_0s]u^{(k)}\xi^2dxds
			%\end{gathered}
		\end{multline}
		where we denote by $u_0 := \sum_{i=1}^{N}u_{i0}$. Since $ \pa_t \wh{u} = \sum_{i=1}^{N}d_iu_i \geq 0$, due to \eqref{uhat_def}, and thus
		\begin{equation*}
			\frac{1}{\max d_i}\pa_t \wh u \leq b\pa_t \wh u \leq \frac{1}{\min d_i}\pa_t \wh u.
		\end{equation*}
		Therefore, with $\nu = \frac{1}{\max d_i}$, we estimate
		\begin{equation}\label{mul1}
			\begin{aligned}
			\int_{t_0-\tau_1}^t\int_{B_{\vr_1}}b\pa_t\wh u u^{(k)}\xi^2dx ds&\geq \nu\int_{t_0-\tau_1}^t\int_{B_{\vr_1}}\pa_t \wh u u^{(k)}\xi^2dxds\\
			&= \frac{\nu}{2}\int_{B_{\vr_1}}(u^{(k)}(t))^2\xi(t)^2dx - \nu\int_{t_0-\tau_1}^t\int_{B_{\vr_1}}(u^{(k)})^2\xi\pa_t\xi dxds\\
			&\geq \frac{\nu}{2}\int_{B_{\vr_2}}|u^{(k)}(t)|^2dx - \nu\int_{t_0-\tau_1}^t\int_{B_{\vr_1}}(u^{(k)})^2\xi\pa_t\xi dxds\\
			&\geq \frac{\nu}{2}\int_{B_{\vr_2}}|u^{(k)}(t)|^2dx - \nu \int_{Q(\vr_1,\tau_1)}|u^{(k)}|^2|\xi_t|dxds,
			\end{aligned}
		\end{equation}
		where we used $\xi(\cdot, t_0-\tau_1) = 0$ at second step and $\xi|_{Q(\rho_2,\tau_2)} \equiv 1$ at the third step. By applying Cauchy-Schwarz's inequality, we get
		\begin{equation}\label{mul2}
			\begin{aligned}
			-2\int_{t_0-\tau_1}^t&\int_{B_{\vr_1}}(\nx u^{(k)}\cdot\nx\xi) u^{(k)}\xi dxds \\
			&\leq  \frac 12 \int_{t_0-\tau_1}^t\int_{B_{\vr_1}}|\nx u^{(k)}|^2\xi^2dxds+ 2\int_{t_0-\tau_1}^t\int_{B_{\vr_1}}|u^{(k)}|^2|\nx \xi|^2dxds\\
			&\leq  \frac 12 \int_{t_0-\tau_1}^t\int_{B_{\vr_1}}|\nx u^{(k)}|^2\xi^2dxds+ 2\int_{Q(\vr_1,\tau_1)}|u^{(k)}|^2|\nx \xi|^2dxds.
			\end{aligned}
		\end{equation}
Denoting by $\chi_{\{\wh u > k \}}$ the characteristic function of the set $\{(x,t): \wh{u}(x,t) > k\}$, 
%		$A(k, \vr, t) = \{x\in B_\vr(x_0): w^{(k)}(x,t) > 0\}$ 
		we can estimate
		\begin{equation}\label{mul3}
			\begin{aligned}
			&\int_{t_0-\tau_1}^t\int_{B_{\vr_1}}[u_0 + K_0s] u^{(k)}\xi^2dxds\\
			&= \int_{t_0-\tau_1}^t\int_{B_{\vr_1}}[u_0 + K_0s]u^{(k)}\xi^2 \chi_{\{\wh u > k\}}dxds \\
			&\leq \frac 12\int_{t_0-\tau_1}^t\int_{B_{\vr_1}}|u^{(k)}|^2\xi^2dxds + \frac 12 \int_{t_0-\tau_1}^t\int_{B_{\vr_1}}|u_0 + K_0s|^2\xi^2\chi_{\{\wh u > k\}}dxds\\
			&\leq \max\left\{1/2; \|u_0\|_{L^\infty(\O)}^2 + K_0^2t^2 \right\}\int_{Q(\vr_1,\tau_1)}\left(|u^{(k)}|^2 + \chi_{\{\wh u > k \}}\right)ds,
			\end{aligned}
		\end{equation}
		where we have used $|\xi| \leq 1$ in the last step. 		We now insert the estimates \eqref{mul1}, \eqref{mul2} and \eqref{mul3} into \eqref{mul} to get for all $t\in (t_0-\tau_2, t_0)$
		\begin{multline}\label{est}
		%\begin{gathered}
			\frac{\nu}{2}\|u^{(k)}(t)\|_{L^2(B_{\vr_2})}^2 + \frac{1}{2}\int_{t_0-\tau_1}^t\int_{B_{\vr_1}}|\nx u^{(k)}|^2\xi^2dxds\\
			\leq \max\{\nu;2;\|u_0\|_{L^\infty(\O)}^2 + K_0^2t^2 \}\left[\int_{Q(\vr_1,\tau_1)}\left(|u^{(k)}|^2(|\nx \xi|^2 + |\pa_t\xi| + 1) + \chi_{\{\wh u > k\}}\right)dxds\right].
		%\end{gathered}
		\end{multline}
By adding the inequality $\frac 12 \int_{t_0-\tau_1}^{t}\int_{B_{\vr_1}}|u^{(k)}|^2\xi^2dxds \leq \frac 12 \int_{Q(\vr_1,\tau_1)}|u^{(k)}|^2dxds$ and taking the supremum over $t\in (t_0 - \tau_2, t_0)$, we have with $\tau_2 < \tau_1$ 
		\begin{multline}\label{est1}
			%\begin{gathered}
			\frac{\min\{\nu,1\}}{2}\left(\sup_{t_0 - \tau_2 < t < t_0}\|u^{(k)}(t)\|_{L^2(B_{\vr_2})}^2 + \int_{t_0 - \tau_2}^{t_0}\|u^{(k)}\|_{H^1(B_{\vr_2})}^2ds\right)\\
			\leq \max\{\nu;2;\|u_0\|_{L^\infty(\O)}^2 + K_0^2t^2\} \left[\int_{Q(\vr_1,\tau_1)}\left(|u^{(k)}|^2(|\nx \xi|^2 + |\pa_t\xi| + 1) + \chi_{\{\wh u > k\}}\right)ds\right]
			%\end{gathered}
		\end{multline}
%		with $\kappa = \frac{4}{\max\{\nu; 1\}}\max\{\nu;2;\frac 12\|u_0\|_{L^\infty(\O)}^2\}$. 
Finally, due to the definition of the cut-off function $\xi$, there exists a constant $C\geq 1$ independent of $\vr_i$ and $\tau_i$ such that $|\nx \xi| \leq C(\vr_1-\vr_2)^{-1}$ and $|\pa_t \xi| \le C(\tau_1 - \tau_2)^{-1}$. Noting also $1 \leq (\tau_1 - \tau_2)^{-1}$ since $\tau_1, \tau_2 \in (0,1)$, we get from \eqref{est1} the energy estimate \eqref{local-en}.
\end{proof}

\subsection{Boundary H\"older continuity}\hfill\\
For this part, we only need $\pa\O$ to be Lipschitz continuous. The proof in this part follows from \cite{Nit11}, where the author applied it to elliptic equations. For any $z\in \pa\O$, we can find an orthogonal matrix $\mathcal O$ and a Lipschitz continuous function $\psi: \R^{n-1} \to \R$ and
\begin{equation*}
	G:= \{(y, \psi(y)+s): y\in B(0,r), s\in (-r,r) \}
\end{equation*}
such that
\begin{equation*}
	\mathcal O (\O - z) \cap G = \{(y, \psi(y)+s): y\in B(0,r), s\in (0,r)\}.
\end{equation*}
Without loss of generality, we can assume that $z = 0$ and $\mathcal O = \mathbb I_d$.
Define the mapping $T: B(0,r)\times (-r,r) \mapsto G$ by $T(y,s) = (y, \psi(y)+s)$, then $T$ is bi-Lipschitz continuous. Using the mapping $S: G \mapsto G$ with $S(T(y,s)) = T(y,-s)$ we can define the reflection mapping of any function $w: (D\subset G)\times (0,T) \mapsto \R$ by
\begin{equation*}
	\wt{w}(x,t):= \begin{cases}
		w(x,t) &\text{ if } x\in G\cap \O,\\
		w(Sx, t) &\text{ if } x\in G \setminus \overline{\O}.
	\end{cases}
\end{equation*}
{\color{black}
\begin{lemma}[H\"older continuity at boundary]\label{Holder-boundary}\hfill\\
	Let $\wh{u}$ solve equation \eqref{uhat}. Then $\wt{u}:= \wt{\wh{u}}$ is the weak solution to
	\begin{equation}\label{reflect}
		\wt{b} \pa_t \wt{u} - \mathrm{div}(A(x)\nabla \wt{u}) = \wt{u}_0 + K_0t
	\end{equation}
	where $A(x) \in \mathbb R^{n\times n}$ is uniformly bounded and positive definite for all $x\in G$, i.e. $\|A(x)\|\leq C$ and there exists $\lambda > 0$ such that $\xi^\top A(x) \xi \geq \lambda |\xi|^2$ for all $x\in G$ and $\xi \in \mathbb R^n$. Therefore,
	$\wt{u} \in C^{\alpha}(\wh{G}\times [t_0,T])$ for any $t_0\in (0,T)$ and $\wh{G} \subset \subset G$. Moreover, 
	\begin{equation*}
		\|\wt{u}\|_{C^{\alpha}(\wh{G}\times [t_0,T])} \leq C_\delta\|\wh{u}\|_{L^\infty(Q_T)}
	\end{equation*}
	where $C_\delta$ depends on $\delta = \mathrm{dist}(\wt{G},\pa G)$.
	Consequently,
	\begin{equation*}
		\|\wh{u}\|_{C^{\alpha}(\wh{G}\cap \O\times [t_0,T])} \leq C_\delta\|\wh{u}\|_{L^\infty(Q_T)}.
	\end{equation*}
\end{lemma}
\begin{proof}
	The validity of \eqref{reflect} can be directly verified using properties of refection functions provided in \cite[Lemmas 3.3 and 3.5]{Nit11}. Note that $\wh{b}$ always satisfies the bound \eqref{bbound}. Therefore, we can repeat the arguments in Subsection \ref{sec:loHo} to get the desired local H\"older continuity of $\wt{u}$, which in turn implies the local H\"older continuity of $\wh{u}$ at the boundary.
\end{proof}}

\subsection{H\"older continuity of $\wh u$}\label{sec:A3}
\begin{proof}[Proof of H\"older continuity in Lemma \ref{lem:uhat}]
If $\Omega = \mathbb R^n$, then the H\"older continuity of $\wh u$ follows directly from Lemma \ref{Holder-local}.

When $\Omega\subset \mathbb R^n$ bounded with Lipschitz boundary $\partial\Omega$, H\"older continuity \eqref{uhat_Holder} up to boundary for $\wh u$ follows from Lemmas \ref{Holder-local} and \ref{Holder-boundary}, thanks to the fact that $\pa\O$ is compact and can be covered by finitely many balls with radius $\rho_0$ (where $\rho_0$ is in Lemma \ref{Holder-local}).
\end{proof}

\medskip
\par{\bf Acknowledgements:} The authors would like to thank Prof. Philippe Souplet for fruitful discussions and especially for pointing out the $L^\infty$-convergence in Theorem \ref{thm:quad}.

%\blue{Thanks to J. Morgan for the estimate concerning $\Delta v_d$.}

This work is partially supported by the International Training Program IGDK 1754 and NAWI Graz.

\end{document}